\documentclass[11pt]{amsart}
\usepackage{amsmath}
\usepackage{amssymb}
\usepackage{color}

\newcommand{\ep}{{ \varepsilon  }}

\newcommand{\bq}{\begin{equation}}
\newcommand{\eq}{\end{equation}}
\newcommand{\pa}{\partial}
\newcommand{\la}{\lambda}
\newcommand{\ga}{\gamma}

\newcommand{\ddt}{\frac{d}{dt}}
\newcommand{\bbr}{{ \mathbb{R}  }}
\newcommand{\na}{\nabla}

\newcommand{\wni}{\langle n_i \rangle}
\newcommand{\wki}{\langle k_i \rangle}
\newcommand{\wpi}{\langle p_i \rangle}

\begin{document}
\bibliographystyle{plain}


\newcommand{\kwon}[1]{{\color{blue} #1  }}

\newcommand{\noi}{\noindent}       
\newcommand{\Z}{\mathbb{Z}}
\newcommand{\R}{\mathbb{R}}
\newcommand{\C}{\mathbb{C}}
\newcommand{\T}{\mathbb{T}}
\newcommand{\bul}{{}}

\newcommand{\N}{\mathcal{N}}
\newcommand{\RR}{\mathcal{R}}
\newcommand{\D}{\mathcal{D}}
\newcommand{\HH}{\mathcal{H}}

\newcommand{\al}{a}
\newcommand{\be}{b}
\newcommand{\Om}{\Omega}
\newcommand{\om}{\omega}
\newcommand{\omm}{\Omega}
\newcommand{\s}{\sigma}
\newcommand{\si}{\Sigma}
\newcommand{\ft}{\widehat}
\newcommand{\wt}{\widetilde}
\newcommand{\cj}{\overline}
\newcommand{\dx}{\partial_x}
\newcommand{\dt}{\partial_t}
\newcommand{\dd}{\partial}
\newcommand{\invft}[1]{\overset{\vee}{#1}}
\newcommand{\lrarrow}{\leftrightarrow}
\newcommand{\embeds}{\hookrightarrow}
\newcommand{\LRA}{\Longrightarrow}
\newcommand{\LLA}{\Longleftarrow}

\newcommand{\wto}{\rightharpoonup}
\newcommand{\Omn}{\Om ({\bf n})}
\newcommand{\edeno}{  \ep^2  N_{\infty}^{4A^2}N_{\al+1}^m C^{A/m}}
\newcommand{\deno}{  N_{\infty}^{4A^2}N_{\al+1}^m C^{A/m}}
\newcommand{\mini}{ |\bn _-|}
\newcommand{\nal}{N_{\al} \le |\bn _-| < N_{\al+1}}
\newcommand{\subal}{{\substack{3\\ \nal}}}
\newcommand{\nalty}{N_{\al+1} \le |\bn_-| <  N_{\infty}}
\newcommand{\subalty}{{\substack{3\\ \nalty}}}

\newcommand{\pim}[2] { \Pi_{j= 1}^ {{#1}} ( |m_j| \wedge N_{#2})^{2s_1}N_{#2}^{2\tau} }
\newcommand{\pin}[2] { \Pi_{j= 1}^ {{#1}}( |n_j| \wedge N_{#2})^{s_1}N_{#2}^{\tau}}

\newcommand{\I}[2]{{\mathbf I}_{({#1},{N_{#2}})}}
\newcommand{\Q}[2]{{\mathbf Q}_{({#1},N_{#2})}}
\newcommand{\Qq}[2]{{\mathbf Q}_{4,({#1},N_{#2})}}

\newcommand{\cobm} { {\mathbf I}_{({\bf m},N_\al)} }
\newcommand{\cobmb} { {\mathbf I}_{({\bf m},N_\be)} }
\newcommand{\cobmu} { {\mathbf I}_{({\bf m},N_{\al+1})} }
\newcommand{\cobmp} { {\mathbf I}_{({\bf m'},N_\al)}}
\newcommand{\cobnh} { {\mathbf{Q}_{4,({\bf n},N_\al) }}}
\newcommand{\cobnhb} { {\mathbf{Q}_{4,({\bf n},N_\be) }} }
\newcommand{\cobnph} { {\mathbf{Q}_{4,({\bf n'},N_\al) }} }
\newcommand{\cobl} {{\mathbf I}_{({\bf l},N_\al)} }
\newcommand{\cobn} {{ \mathbf Q}_{({\bf n},N_\al)}}
\newcommand{\qa}[1] {{ \mathbf Q}_{( {#1},N_\al)}}
\newcommand{\ia}[1]{{ \mathbf I}_{( {#1},N_\al)}}
\newcommand{\nnp}{ ( \bn, \bnp)}
\newcommand{\mmp}{ (\bm, \bmp)}
\newcommand{\msn}{ (\bm, \bm^{\sim n})}
\newcommand{\nsn}{ (\bn^{\sim n}, \bn^{\sim n})}
\newcommand{\cobnb} {{ \mathbf Q}_{({\bf n},N_\be)}}
\newcommand{\cobnp} {{ \mathbf Q}_{({\bf n'},N_\al)}}
\newcommand{\cobnu} {{ \mathbf Q}_{({\bf n},N_{\al+1})}}
\newcommand{\dec}{N_{\al}^{-4s_1}}
\newcommand{\hdec}{N_{\al}^{-2s_1}}

\newcommand{\lbm} {|{\bf m}|}
\newcommand{\lbn} {|{\bf n}|}
\newcommand{\lbmp} {|{\bf m'}|}
\newcommand{\lbnp} {|{\bf n'}|}
\newcommand{\lbl} {|\bf l|}
\newcommand{\io}[1]{I_{n_{#1}}^0}
\newcommand{\ii}[1]{I_{n_{#1}}}
\newcommand{\Ibn}{I_{\bn}}
\newcommand{\ibn}{I_{\bn}^0}
\newcommand{\jb}[1]
{\langle #1 \rangle}
\newcommand{\cmnIQ}{c_{{\bf m}{\bf n}} I_{\bf m} q_{\bf n} }
\newcommand{\cmI}{c_{{\bf m}} I_{{\bf m}}}
\newcommand{\cmn}{c_{{\bf m \bf n}}}
\newcommand{\dmn}{d_{{\bf m \bf n}}}
\newcommand{\jm}{I_{\bf m}}
\newcommand{\qn}{q_{\bf n}}
\newcommand{\bmn}{b_{\bm \bn}}
\newcommand{\amn}{a_{\bm\bn}}
\newcommand{\bm}{ {\bf m}}
\newcommand{\bn}{ {\bf n}}
\newcommand{\bk}{ {\bf k}}
\newcommand{\bp}{ {\bf p}}

\newcommand{\bmp}{ {\bf{ m}'}}
\newcommand{\bnp}{{\bf{n}'}}
\newcommand{\bpp}{{\bf{p}'}}
\newcommand{\bkp}{{\bf{k}'}}
\newcommand{\jmp}{I_{\bf{ m}'}}
\newcommand{\qnp}{q_{\bf{ n}'}}
\newcommand{\jl}{I_{\bf l}}
\newcommand{\bl}{{\bf l}}
\newcommand{\cmnp} {c_{\bm' \bn'}}

\newcommand{\PI}{ {\mathbf I}_{({\bf m},N_\al)}}
\newcommand{\Pq}{ {\mathbf Q}_{({\bf n},N_\al)}}
\newcommand{\Pqq}{{\mathbf{Q}_{4,({\bf n},N_\al) }}}

\newtheorem{defn}{Definition}
\newtheorem{lemma}[defn]{Lemma}
\newtheorem{proposition}{Proposition}
\newtheorem{theorem}[defn]{Theorem}
\newtheorem{cor}{Corollary}
\newtheorem{remark}{Remark}
\numberwithin{equation}{section}

\newenvironment{pf}{{\bf Proof.}}{\hfill\fbox{}\par\vspace{.2cm}}
\newenvironment{pfthm1}{{\par\noindent\bf
            Proof of Theorem 1. }}{\hfill\fbox{}\par\vspace{.2cm}}
\newenvironment{pfthm2}{{\par\noindent\bf
            Proof of Theorem 2. }}{\hfill\fbox{}\par\vspace{.2cm}}
\newenvironment{pfthm3}{{\par\noindent\bf
Proof of Theorem 3. }}{\hfill\fbox{}\par\vspace{.2cm}}
\newenvironment{pfthm4}{{\par\noindent\bf
Proof of Theorem 4. }}{\hfill\fbox{}\par\vspace{.2cm}}
\newenvironment{pfthm5}{{\par\noindent\bf
Proof of Theorem 5. }}{\hfill\fbox{}\par\vspace{.2cm}}
\newenvironment{pflemsregular}{{\par\noindent\bf
            Proof of Lemma \ref{sregular}. }}{\hfill\fbox{}\par\vspace{.2cm}}

\title[Stability of NLS in high Sobolev norm]{{The stability of nonlinear Schr\"odinger equations with a potential in high Sobolev norms revisited}}
\author[M. Chae]{Myeongju Chae}
\address{ Department of Applied Mathematics, Hankyong National University, Ansong, Republic of Korea}
\address
{Department of Mathematics, University of Wisconsin-Madison, USA} \email{mchae@hknu.ac.kr}
\author[S. Kwon]{Soonsik Kwon}
\address{Department of Mathematical Sciences, Korea Advanced Institute of Science and Technology, Daejeon, Republic of Korea} \email{soonsikk@kaist.edu}

 \begin{abstract}
 {{We consider the nonlinear Schr\"odinger equations with a potential on $\mathbb T^d$. For almost all potentials, we show the almost global stability in very high Sobolev norms. We apply an iteration of the Birkhoff normal form, as in the formulation introduced by Bourgain \cite{Bo00}. This result reprove a dynamical consequence of the infinite dimensional Birkhoff normal form theorem by Bambusi and Grebert \cite{BG} }
 }
 
\end{abstract}

\maketitle

\section{Introduction}
We consider the nonlinear Schr\"odinger equation with a potential $V$ 
\begin{align}\label{nls}
 \qquad&  iu_t = -\Delta u + V\ast u + 2|u|^2u \quad x\in \mathbb T^d, t\in \bbr, \\
&u(0,x) = u_0(x) \, \in H^s(\mathbb T^d) \nonumber
\end{align}
where $V$ is a smooth convolution potential. 
The system \eqref{nls}  is an infinite dimensional Hamiltonian system associated with the Hamiltonian functional
\[ H(u, \bar u) = \int_{\mathbb T^d} |\na u|^2 + (V\ast u) \bar u +  |u|^4 dx .\]

The aim of this work is to use a Hamiltonian dynamical method for studying the long time stability of small solutions in $H^s(\mathbb T^d)$ for a sufficiently large $s$.


In this paper we consider $V$ a random potential
\[ V(x) = \sum_{n\in \Z} v_n e^{inx}, \quad v_n  = R(1+ |n|)^{-m} \sigma_n, \] 
$\{\sigma_n\}$ being a sequence of i.i.d. random variables uniformly distributed over $[-\frac 12, \frac 12]$. In usual NLS, the potential $V(x)$ is multiplicative type, but here we choose the convolution type potential so that the formalism is simpler in Fourier variables. We assume that the potential $ V$ is an even function to ensure its Fourier coefficient $v_n \in \bbr$. We define the measure space for potential: 
\begin{equation}\label{dataset}
  \mathcal W = \{ V= \sum_{n\in \Z} v_n e^{inx}|\sigma_n:= R^{-1} v_n (1 + |n|)^{m}\in [ -\frac{1}{2}, \frac{1}{2}]\},
\end{equation}
for some $m \in \Z_+$. We endow $\mathcal W$ with the product probability measure, that is, 
\begin{align*}
 \mathbb P \big( &\{ W= \sum v_n e^{inx} \in \mathcal W |  \sigma_{n_i} \in (a_i, b_i)\subset
 [-1/2, 1/2]  \mbox { for } i = 1, \dots N ,
\mbox{ otherwise } w_n = 0\} \big) \\
 &= \prod_{i=1}^N (b_i -a_i).
\end{align*}

As we handle small solutions, we rescale $u(t,x)$ by 
\bq
\label{scaling} u(t,x)=  \ep q(\ep^2 t, x) = \ep\sum q_n(\ep^2t) e^{inx}.
\eq
We consider for initial data $q_0(x) = q(0, x)$  of $O(\ep)$ in $H^s( \mathbb T^d)$, 
and so $ \| u_0\|_{H^s} \le O(\ep^2)$. Then, the equation \eqref{nls} becomes 
\[  iq_t = -\ep^{-2}\Delta u + \ep^{-2}V\ast u + |q|^2q \]
with Hamiltonian 
\[ H(q, \bar q) = \int_{\mathbb T^d} \ep^{-2} |\na q|^2 + \ep^{-2}(V\ast q) \bar q +  |q|^4 dx.\]
One can rewrite  the infinite dimensional Hamiltonian system in Fourier variable $\{q_n\}$:
\begin{equation}\label{qn}
i\dot q_n = \frac{\pa H}{\pa \bar q_n}
\end{equation}
where 
\begin{equation} \label{hqn}
H(q, \bar q) = \sum ( \frac{|n|^2}{\ep^2}  +\frac{v_n}{\ep^2})|q_n|^2 + \sum_{n_1-n_2+n_3-n_4 = 0}
q_{n_1}\bar q_{n_2}q_{n_3}\bar q_{n_4}.
\end{equation}
We state our main theorem.
\begin{theorem}\label{main}
There exists a subset $\mathcal V \subset \mathcal W$ of full measure, such that for  a given $V
\in \mathcal V$ the following holds; for a given $B>0$ there exist $C$, $s$, and $\ep(B) >0$ such that 
if $\| u(0)\|_{H^s} \le \ep^2$, the solution $u$ to the Cauchy problem
\[iu_t = -\Delta u + V\ast u + |u|^2u, \quad x\in \mathbb T^d, t\in \bbr\]
will satisfy 
\[ \sup_{|t|<T} \|u(t)\|_{H^s} < 2\ep^2\]
where $T$ can be as large as $\ep^{-B}$.
\end{theorem}
We remark that the theorem holds true for $|u|^{2p}u$ for any  $p\ge 1$.
The analysis is similar, thus for simplicity we present $p=1$ case in this paper. \\
\indent
Theorem \ref{main} is a version of \textit{Birkhoff normal form theorems} for the infinite dimensional Hamiltonian system: the Hamiltonian flow associated to a Hamiltonian remains close to 
the initial state during an arbitrarily long polynomial time ($ \ep^{-B}$), if the initial state has a sufficiently small by $\ep$ in $H^s(\mathbb T^d)$. Some of  instructive expositions for the {\it Birkhoff normal form}  theory can be found in \cite{Ba03, G00, KP}. There are other stability results on \eqref{nls} such as the 
KAM theorem by Elliason-Kuksin \cite{EK} with analytic potentials $V$, or Nekhorochev type theorem by Faou-Grebert \cite{FG} with analytic data. 
{Theorem \ref{main} reproves a dynamical consequence of infinite dimensional Birkhoff normal form theorem by Bambusi-Grebert \cite{BG}. In \cite{BG}, the authors construct an abstract Birkhoff normal theorem to infinite dimensional Hamiltonian systems and apply it to PDEs with tame modulus. \cite{BG} is systematic and applicable to a wide range of PDE examples. \\ 
\indent
In this work, we revisit the problem with a more direct approach to the equation. In performing the Birkhoff normal form, we would like to track how the Hamiltonians are changed. Using a sequence of frequency cut-off we obtain a concrete information on the final Hamiltonian to exhibit the stability result. In fact, we are inspired by Bourgain \cite{Bo00}, and this work follows a similar line of \cite{Bo00}. 
 \\ } 
\indent
In \cite{Bo00} Bourgain consider one dimensional  Schr\"{o}dinger equations with random initial data,
\bq \label{NLS}
 \qquad\quad iu_t = - u_{xx}  + |u|^2u + \la |u|^{2p} u, \quad x\in \mathbb T, \, p >1,
 \eq 
When $\la =0$ \eqref{NLS} has been known to be integrable. In \cite{GKP} the authors proves that the global Birkhoff coordinate exists. In \cite{Bo00} Bourgain proves that for given $B>0$, for almost all initial data (with probability one) the solutions are stable up to time $ \epsilon^{-B}$. The work of \cite{Bo00} does not rely on integrability, however the presence of cubic term $|u|^2u$ is essential. 
As like \cite{BG} and \cite{Bo00} uses Birkhoff normal form, a nonlinear change of coordinate of \textit{symplectic transform} to reduce the \textit{non resonant} terms from the given Hamiltonian. We use the formulation of the sequence of Birkhoff normal form of Bourgain's to obtain a similar result for \eqref{nls}. In use of normal forms, the nonresonancy is inherited from the randomness of initial data as opposed to in \eqref{nls}, where the nonresonancy condition is from random potentials. Indeed, the randomness is explicit in the Hamiltonian \eqref{hqn} for the \eqref{nls} due to the random potential $V$.  If we define 
 $\omega_n  =  |n|^2/\ep^2 + v_n/\ep^2$,  the denominator arising in normal form transform 
 is of the form 
 \[ \Om( \bn) = \om_{n_1} - \om_{n_2} + \cdots  \pm \om_{n_r }\quad \bn= (n_1, n_2, \dots, n_r).\]
As a similar argument to \cite{BG} (or \cite{FG}), we obtain the lower bound estimate
\bq\label{lb-nlsv}
 |\Omega ({\bf n})| \ge
 \ga R \ep^{-2} C^{-r/m}
 \mu({\bf n}) ^ {- {4r^2}}  n_- ^ { - m}
\eq 
for most of potentials $V$. \footnote{
 See a precise probabilistic statement in Lemm \ref{lblemma}.}  The lower bound here depends on $\mu(\bn)$: the third biggest entry among $|n_j|'s$ where $ \bn = (n_1, n_2, \dots, n_r)$, and $n_-$ denotes the least entry. For \eqref{NLS} the randomness is given to the initial data by $q_n(0) = \ep (1 + |n|)^{-(1+s)} \sigma_n $.
Indeed, rescaling \eqref{NLS} by \eqref{scaling}, one can write the associated Hamiltonian as
\begin{align*} 
H &=   \sum \ep^{-2}  |n|^2|q_n|^2  + \sum_{n_1-n_2+n_3-n_4 = 0}
q_{n_1}\bar q_{n_2}q_{n_3}\bar q_{n_4}\\
& = \sum \ep^{-2}  |n|^2|q_n|^2 + 2 \left( \sum |q_n|^2 \right)^2 - \sum |q_n|^4 
+ \sum_{\substack{n_1-n_2+n_3-n_4 = 0\\ n_1^2 - n_2^2 + n_3^2 - n_4^2 \neq 0 }}
q_{n_1}\bar q_{n_2}q_{n_3}\bar q_{n_4}.
\end{align*}
The latter equality follows from that all the resonant terms of $q_{n_1}\bar q_{n_2}
q_{n_3}\bar q_{n_4}$ are fully resonant on $\mathbb T$.\footnote{
If $\bn= (n_1, n_2, n_3, n_4)$ $n_i\in \mathbb Z$ satisfies
$ n_1-n_2+n_3-n_4 =  n_1^2 - n_2^2 + n_3^2 - n_4^2=0$, it implies
 $\{ n_1, n_3\} = \{n_2, n_4\}$.}
If we replace $|q_n|^2$ by  $ J_n = |q_n|^2 - |q_n(0)|^2$,  the randomness comes into play in the Hamiltonian: 
\[ H 
=  \sum \underbrace{(n^2/\ep^2 -2|q_n(0)|^2)}_{\om_n}J_n + J_n^4+  \sum_{\substack{n_1-n_2+n_3-n_4 = 0\\ n_1^2 - n_2^2 + n_3^2 - n_4^2 \neq 0 }}
q_{n_1}\bar q_{n_2}q_{n_3}\bar q_{n_4}.\]
In \cite{Bo00} the lower bound estimate of $\Om(\bn)$, 
\begin{equation}\label{lb-nls} 
\Om(\bn) \ge \ep^2 {n^*_1}^{-5r^2}n_-^{-2s}
\end{equation}
holds with large probability, where
$n_1^*$ is  the biggest entry of $\bn = (n_1, n_2, \dots, n_r)$. 
Note that the right hand side has  also the factor $n_-^{-2s}$.
 Because $s$ is chosen to be large
for the perturbation terms of Hamiltonian to be small enough,
the lower bound of \eqref{lb-nls} becomes smaller as increasing $s$.
This small denominater issue can be overcome
if coefficients of perturbation terms are appropriately small.
In \cite{Bo00} the author performed the normal form 
transformation to \eqref{nls} inductively to have the series of Hamiltonians and  to reach 
the final one, for which coefficients are small as desired.
Once the right induction hypothesis are assumed on the size of coefficients of polynomials in Hamiltonian, the consequential analysis goes straightforward in \cite{Bo00}.\\    {
\indent In this paper, we apply the technique in \cite{Bo00} to higher dimensional case $\mathbb T^d$. As opposed to one-dimensional case, the lower bound of the small denominator \eqref{lb-nlsv} is involved in the third largest frequency $\mu(\bn)$. For 1 dimensional  \eqref{nls} as well as \eqref{NLS} there is no difference in analysis if 
 $n_1^*$ and  $\mu(\bn)$ are replaced with each other in the lower bound estimates \eqref{lb-nlsv} or \eqref{lb-nls}.
It is due to that
$n_1- n_2 + \cdots = 0$ implies  $ \mu(\bn) \gtrsim (n^*_1)^{1/2}$ on $\mathbb T^1$. But this is no longer true for $\mathbb T^d$, $d>1$. (See the estimates around \eqref{bmn-part} ).
%
Another aspect of our approach is that we are able to see the regularity of the potential  $V$ with respect of $B$ and $s$. Indeed, $m$ is less than $O(\frac s B)$ and  $s$ is  bigger than $B^3$.}
  \\
 \indent
 We  mention that the abstract theorem in  \cite{BG} is applied to several other equations
 than \eqref{NLS} 
 to obtain Birkhoff normal form theorems.
It might be possible that the inductive use of normal form transform in \cite{Bo00} can be 
applied to reprove the known results on Birkhoff normal form theorems. We have not pushed 
in this direction, however the method seems quite robust.  To our knowledge the similar use of  iterative Birkhoff normal form transforms is found in \cite{Wa}.  In \cite{Wa} Wang proved a long time Anderson localization for the 1-d lattice nonlinear random Schr\"{o}dinger equation. We also remark that in \cite{CHL} 
Cohen,Hairer, and Lubich proved a long time stability result for 1-d nonlinear Klein-Gordon equation via modulated Fourier expansion method without using Birkhoff normal form.
%
\\
\\
The paper is organised as follows: In Section 2, we state preliminary setting of Hamiltonian systems and Birkhoff normal form as well as the estimate of the denominator. Section 3 includes the main analysis of the Birkhoff normal form. We present the reduction of Hamiltonian and the estimate of coefficient of them. In Section 4, we provide the proof of main theorem.
\subsection*{Notations} $ $\\
We abuse multi indices notation in \emph{bold}.
\begin{align*}
n &= (n^1, \dots, n^d) \in \mathbb Z^d, \quad |n|^2  = |n^1|^2 + \cdots + |n^d|^2 \\
\bf m &= (m_1, \dots, m_l)\in  \underbrace{\mathbb Z^d \times \cdots \times \mathbb Z^d}_{l\mbox{ times}} \\
\bf n  &=  (k_1, \dots, k_L, p_1, \dots, p_L) =({\bf k, \bf p})  \in 
\underbrace{\mathbb Z^d \times \cdots \times \mathbb Z^d}_{2L \mbox{ times}},
 \\
 |{\bf m}|& = l, \quad |\bn| = 2L \\
I_{\bf m} & = I_{m_1}\cdots I_{m_l}, \quad
q_{\bf n} = q_{k_1}q_{k_2}\cdots q_{k_l} \bar q_{p_1}\cdots \bar q_{p_l}\\
\omega_n & =  \frac{|n|^2}{\ep^2} + \frac{v_n}{\ep^2}.
\end{align*}
We say $ m\in \bf m$ if $m= m_j$ for some $j$. Similarly $n \in \bm \cap \bn$
if $n \in \bm$ and $n\in \bk$ or $n \in \bp$.  In the above notations, $| \bf n| $ 
denotes the degree of $q_{\bn}$, not the length of the vector-valued index $\bm$.
 The juxtaposition of two multi-index is written as
 $(\bm, \bm')$ \textit{i.e.} $(\bm, \bm')= (m_1,m_2, \dots, m_l, 
m_1', \dots, m_k')$ when $ \bm =(m_1, \dots, m_l) $ and $\bm'=
(m_1', \dots, m_k')$. Also we denote
\begin{align*} 
\bn_+& = \mbox{ the biggest entry among }  |n_j|'s \\
\bn_-& = \mbox{ the least entry among }  |n_j|'s \\
 \mu (\bf n) &= \mbox{ the third biggest entry among }\{ |n_j| | j=1, \dots 2L\}, \\
 \Om (\bf n) &= \sum _{n\in \bf k} \omega_n  - \sum_{n\in \bf p} \om_n.
 \end{align*}
 On account of  $\sigma_n = R^{-1}v_n(1+ |n|)^m$, we write
\[\Om({\bf n})= \ep^{-2} [ \sum_{i=1}^L(  k_i^2 +R^2(1 + |k_i|)^{-m} \sigma_{k_i})
 -  \sum_{i=1}^L ( p_i^2 +R^2(1 + |p_i|^2)^{-m} \sigma_{p_i})]
\]
 $\bf n$ is called {\bf{non-resonant}} if   $ (k_1, \cdots, k_L) $ is not equal to $( p_1, \cdots, p_L)$ by a permutation.
Sometimes we denote the largest entry of $\bn$ by  $n_1^*$, and the next biggest entries by $|n_2^*| \ge
|n_3^*| \dots$ etc.

\section{Preliminaries}
\begin{subsection}{Symplectic transfomations}
We briefly review  basic definitions on the infinite dimensional
Hamiltonian system. In practice what we will use in the sequel is the equations \eqref{fmap} -\eqref{defF}.
For more details we refer to \cite{G00}, \cite{KP}.

The phase space $\mathcal P$ is defined by
$$ \quad \mathcal P_s : =  l_s^2 (\mathbb C) \times   l_s^2(\mathbb C),
\mbox{ where }
 l_s^2 (\mathbb C) : = \{ (q_n) \in \mathbb C^{\mathbb Z^d}| \sum |n|^{2s} |q_n|^2< \infty\}.$$
We identify $ q \in H^s(\mathbb T^d)$  with  $(q_n) \in l_s^2 $  by $ q = \sum q_n e^{i n\cdot x}$
and call $\, ( q, \bar q)$ a canonical coordinate of $\mathcal P_s$.
 We endow $\mathcal P_s$  with  
the symplectic 2 form
\[i \sum dq_n \wedge  d \bar q_n,  \]
which induces  the symplectic operator $J$,  Poisson bracket $\{ \, , \, \}$ as follows,
\[ i \sum_{n\in \mathbb Z^d} dq_n \wedge  d \bar q_n ( v, w)  = \langle v, J^{-1} w \rangle,\]
\[
\{ F,  G\} = i \sum_n  \frac{\pa F}{\pa q_n}\frac{\pa G}{\pa \bar{q}_n}
- \frac{\pa F}{\pa \bar{q}_n}\frac{\pa G}{\pa q_n}.
\]
A smooth function  $ F\in C(\mathcal P_s, \mathbb C)$ is called a Hamiltonian. 
The  Hamiltonian vector field associated to $F$  is defined by
\[ X_F = J \na F = \big(-i \frac{\pa F}{ \pa \bar q}, i \frac{\pa F/ }{\pa q }\big)^{T},\]
and the Hamiltonian flow 
associated to  F  by the integral curve  $(q(t), \bar q(t))$ along $J\na F$ such that
 $(q(t), \bar q(t))$  satisfies the ODE
\[ i \ddt (q, \bar q)^T = X_F (q, \bar q).\]
In terms of coordinate $(q_n,  \bar q_n)$ it is written
\begin{equation}\label{odeF}
i \dot{q_n} = \frac{\pa F} {\pa \bar q_n}, \quad n\in \mathbb Z^d.
\end{equation}
\indent
Next we introduce the symplectic transformations.
A  diffeomorphism
  $\varphi : \mathcal P_s \to \mathcal P_s $  is called  \textit{symplectic transformation} if $\varphi$ preserves the 
 Poisson bracket
 \begin{equation}\label{preserve}
  \{ F \circ \varphi, G \circ \varphi \} = \{ F, G\} \circ \varphi.
  \end{equation}
Symplectic transformations  preserve the flow law, that is, if $(q, \bar q)$ is the Hamiltonian flow associated to $H$,  the new coordinate $ (q', \bar q')$ given by
$(q', \bar q') \substack{ \varphi \\ \to} (q, \bar q )$  satisfies the following system of ODE,
\[  i \ddt{ q'_n} = \frac{\pa  H'} {\pa \bar{{q'_n}}} , \quad  H'  = H\circ{ \varphi}.\]\indent
What it follows we consider the symplectic transformation, \textit{time 1-shift}.
For a given Hamiltonian $F$ let us 
consider the  Hamiltonian flow generated by F,
and denote the solution at time $1$ by $q_n$ 
\bq\label{fmap}
 \dot{ {q_n }} = i \frac{\partial F}{\partial \bar{  { q_n}}}, \quad  q_n(0) = q'_n, \quad 
   q_n(1) := q_n.
   \eq
 The map $ q_n' \to q_n $ is called  \textit{ time 1 shift} by $F$, which is denote by $\Phi_F$. The map ${ \Phi^t_F}: { q_n}(0)\to { q_n} (t)$ is defined similar way. 
 It is known that $ \Phi^t_F$  is  symplectic  if
 the ODE \eqref{fmap} is well-posed on $[0, t]$.\\
 \indent
 The Hamiltonian is shifted to $H\circ \Phi_F$ by the coordinate change $\Phi_F$.
 We note that by \eqref{odeF}, \eqref{preserve} and the chain rule it  holds that
\[\frac{d}{dt} (H\circ \Phi^t_F) = \{ H, F\} \circ \Phi^t_F \quad \mbox{ for any } H.\]
Applying the taylor series expansion centered at $t=0$, we have the  following expression of
$H\circ {\Phi_F}$,
\begin{equation}
\label{recall} H\circ {\Phi_F} = \sum_{k = 0}^\infty \frac{1}{k!} \{ \cdots \{ H, 
\underbrace{ F\}, F\}, \cdots, F\} }_{k \text{ times}}
\end{equation}
 \[= H + \{ H, F\} + \frac{1}{2!}\{ \{H, F\}, F\} +  h. o. t. \]
 \indent 
 Now we demonstrate how to reduce a lower order polynomials of given Hamiltonian using 
 a time 1-shift.
Back to the Hamiltonian 
 \[H=\underbrace{ \sum \om_n |q_n|^2}_{H_0} + \underbrace{\sum_{l(\bn) =0}a_{\bn} q_{n_1}\bar q_{n_2} q_{n_3}\bar q_{n_4}}_{H_1},\]
 we have a shifted Hamiltonian
\begin{align*}
 H\circ \Phi_F &= (H_0 + H_1) \circ \Phi_F  \\
& = H_0 +\underbrace{ H_1 +  \{ H_0, F\}}_{\mbox{ degree of } 4} +  \{ H_1, F\}
 + \text{h.o.t.}
\end{align*}
by a time $1$-shift by $F$.
If  we choose
\begin{equation}\label{defF}
F = \sum_{l(\bn) =0, \Omega(\bn) \neq 0}i \frac{a_{\bn}}{ \Omega(\bn)} q_{n_1}\bar q_{n_2} q_{n_3}\bar q_{n_4}, \mbox{ when }  \Omega(\bn) = \om_1 - \om_2 + \om_3 - \om_4,
\end{equation} it is straightforward to compute
\[ \{H_0, F\} = -(H_1 - \sum_{\Om(\bn)=0} a_\bn q_{n_1} \bar q_{n_2} q_{n_3} \bar q_{n_4} ).\]
We can only  reduce 'non resonant' monomials of $\Omega(\bn) \neq 0$, meanwhile,  there are abundant resonant momonials in $H$.  This is where the randomness comes to play by modulating 
frequency $\Om({\bf n})$ so that the denominator is away from zero at a large probability.
\end{subsection}
\begin{subsection}{The lower bound of the denominator } $ $ \\{
In performing the Birkhoff normal form, we should know that the Hamiltonian has good behaved to the Poisson bracketing. As one see from \eqref{defF}, we require a lower bound of the denominators, $ \Omega(\bn)$,  to satisfy \textit{so called} the strongly non resonant condition. The following lemma guarantees the strongly non resonant condition is rather \emph{generically} satisfied.}
\begin{lemma}\label{lblemma}
Fix $0<\gamma <1$ small enough, and $\Om(\bn)$, $V$, and $\mathcal W$ are the same as above.
  There exists a set $F_{\gamma} \subset \mathcal W$ whose 
measure is larger than $1-\gamma$ such that  if $ $V$ \in F_{\gamma}$ then 
\begin{equation}\label{dsize}
|\Omega ({\bf n})| \ge
 \ga R \ep^{-2} C^{-r/m}
 \langle \mu({\bf n}) \rangle^ {- {4r^2}} \langle \bn_- \rangle^ { - m}
\end{equation}
for all non resonant ${\bf n} = (k_1, \dots, k_r, p_1,\dots, p_r)$ and a constant 
$C=   ( 40/\gamma)^{4}  $.
\end{lemma}
Our strongly nonresonant condition is controlled by the third largest frequency and the lowest frequency, as well as the regularity parameter $m$ of potential space. In one dimensional NLS \eqref{NLS}, we have $\mu({\bf n}) \lesssim (n^*_1)^{1/2} $. Thus, the lower bound may be involved in $n^*_1$. However, in higher dimensional case, this is no longer true. The proof of \eqref{size} is similar to that in Faou and Grebert \cite{FG}. For the convenience of readers, we place it  in the Appendix~\ref{AppB}. \\
\indent
Due to Lemma \ref{lblemma}, if we set $\mathcal V = \cup_{\ga>0} F_{\ga}$, 
then $\mathbb P(\mathcal V) = 1$ and  any $V \in \mathcal V$ satisfies the nonresonant condition 
\eqref{dsize}.
\end{subsection}
\begin{subsection}{The Poisson brackets} $ $ \\
We use the multi index notation as follows:
$$I_{\bf m}  = I_{m_1}\cdots I_{m_l}, \quad
q_{\bf n} = q_{k_1}q_{k_2}\cdots q_{k_l} \bar q_{p_1}\cdots \bar q_{p_l}
, \quad \jl = I_{l_1}\cdots I_{l_j}.$$
 A straightfoward computation shows that 
\begin{align}
&\{ \jm, \jl\} = 0, \quad \{\jm\qn, \jl\} = \jm\{\qn, \jl\}, \nonumber\\
\label{b1}
 &\{ \cmn\jm\qn, \sum \om_l I_l\} 
= (\sum_{i=1}^l \om_{k_i} - \sum_{i=1}^l \om_{p_i})\cmn \jm\qn ,\\ 
\label{b2}
 &\{ \cmn\jm\qn, \sum  I_l^2\} 
= (\sum_{i=1}^l I_{k_i}  - \sum_{i=1}^l I_{p_i})\cmn \jm\qn.
\end{align}
We denote the contraction by 
$ I_{\bf m}^{\sim { m}} = I_{\bf m} / I_m$ for $m \in \bf m$ and  $q_{\bf n}^{ \sim n} = q_{\bf n} / q_n  $ for $n\in \bf k$, or
$q_{\bf n}/ \bar q_n $ if $n\in \bf p$.  
Moreover, when $ \bm = (m_1,\dots, m_{i}, \dots, m_l)$ and $m_i$ is contracted, we 
denote the multi index  after the contraction by $\bm^{\sim m_i}: =  (\dots, m_{i-1}, m_{i+1}, \dots)$.
So $ I_{\bf m}^{\sim { m}}=  I_{\bf m^{\sim { m}}}$ \textit{etc.}
 We denote the number of $n$ appearing in $\bf m$ by 
$\sharp n (\bf m)$ , then compute that  
\[\{ \qn , \jl\} = \sum_{n\in \bk\cap \bl} \sharp n (\bl) \sharp n(\bk) \jl^{\sim n} \qn 
- \sum_{n\in \bp \cap \bl} \sharp n (\bl) \sharp n(\bp) \jl^{\sim n} \qn .\footnote{We introduce $\sharp n(\bf m)$ for the sake of concreteness only.
Mostly we use the upper bound $\sharp n(\bf m) \le \lbm$.} \]
{For simplicity, we slightly abuse notations, writing 
\begin{align}
\label{b3}
 \{ \cmn \jm \qn , a_\bl I_{\bf l} \} &=\ \cmn |\bn| |\bl| \jm \qn
 \sum_{ n\in \bn \cap \bl} a_\bl \jl^{\sim  n}.
\end{align} 
In the sequel, we will estimate the coefficients $\cmn$ for each cases. Thus, the equality means that $\cmn$ of LHS is replaced by new coefficient $\cmn$, (still denoted by $\cmn$), with the same upper bound.} \\
$\{ \jm\qn, \jmp\qnp\}$ give rise to two types, which are occurred from loss of $I_m$ or a pair of $(q_n, \bar q_n)$. As above, we write
\begin{align}\begin{aligned}\label{b4}
\{ \jm\qn,& \jmp\qnp \}  =
& \begin{cases}
 \lbm\lbnp(\sum_{ n\in \bm \cap \bnp} \jm^{\sim  n})
I_\bmp\qn\qnp\\
\lbn\lbmp( \sum_{ n\in \bn \cap \bmp} \jmp^{\sim  n})
\jm\qn\qnp\\\lbn\lbnp
(\sum_{ n\in \bn \cap \bnp}\qn^{\sim  n}\qnp^{\sim  n})\jm\jmp.
\end{cases}
\end{aligned}
\end{align}
\end{subsection}

\section{The reduction of Hamiltonians}
In this section, we discuss how to iterate the symplectic transformations and show that starting from the Hamiltonian \eqref{hqn}, we reduce to the final Hamiltonian $H_b$.  Then, we show  the new Hamitonian flow associated to $H_b$, still denoted by  
$\{q_n(t)\}$,  remains $\ep$-neighborhood of zero for a long time $T$ with $T\sim \ep^{-B}$ for any given $B$.\\
Define the actions of the phase variables. 
\[ I_n^0 = |q_n(0)|^2, \quad I_n(t) = |q_n(t)|^2 =|q_n|^2.\]
Let $N_a$ and $N_\infty$ be cut-off parameters and we set a large parameter $A (>200B)$. In the middle of the reduction procedure, we have the Hamiltonians of the following form:
\begin{align}
H  &= \sum  \omega_n I_n - \sum I_n^2  \qquad &=:\Sigma_0 + \Sigma_1 \nonumber \\
 &+  \sum_{ |\bm_+| < N_{\al}} \cmI   \, +  \, \sum_{ |\bm_+| \ge  N_{\al}} \cmI  \qquad &=:\Sigma_{2_1}
 + \Sigma_{2_2} \nonumber\\
& +  \sum_{N_{\al} \le |\bn_-| \le  \mu (\bn) < N_{\infty}}
 \cmnIQ \qquad &=:\Sigma_3  \nonumber\\
& +  \sum_{ \mu(\bn) \ge N_{\infty} }\cmnIQ \qquad &=:\Sigma_4 \nonumber\\
& +  \sum_{A< \text{deg} \le 2A  }\cmnIQ  \qquad &=:\Sigma_5 \nonumber\\
& + \sum_{ \text{deg} > 2A} \cmnIQ \qquad &=:\Sigma_6 \label{Harrange}\\
&+ \ep^A \sum \cmnIQ \qquad &= : \Sigma_7. \nonumber
\end{align}
Here monomials  of $\Sigma_2$, $\Sigma_3, \Sigma_4$, and $\Sigma_7$  are  of degree $\le A$. 
Moreover, the degree of $\qn$  in $\Sigma_3$ has at least 4.
In $\Sigma_3 \sim \Sigma_7$, $\qn$ is \emph{fully nonresonant} 
in the sense that 
 \[ \{ k_1, k_2 , \cdots\} \cap  \{ p_1, p_2 , \cdots\} = \varnothing.\]
In other words, if it \emph{were} $k_i = p_{j}$, then  the term $q_{k_i}\bar q_{p_{j}} = |q_{k_i}|^2$ already 
makes $I_{k_i}$ and is set aside from $q_\bn$. 
The decomposition is not unique. For example, $I_{\bm}$ can be counted either in 
$\sum_{2}$ or in one of $\sum_{5}, \sum_6, \sum_7$.
We set more parameters
\begin{equation}\label{stau}
s =s_1 + 5\tau, \quad \tau = \frac{10 s}{A},
\end{equation}
where $5\tau$ is a parameter that $\sum_{n\in \mathbb Z^d} 1/|n|^{5\tau} < \infty $, hence 
$5\tau > d$. \eqref{stau} is used in the proof of Proposition \ref{induction3}. \\
\indent
We use an induction argument to prove an iteration of the Birkhoff normal forms changes the initial Hamiltonian \eqref{hqn} to a final Hamiltonian $H_b$. For this purpose, we impose induction hypotheses to coefficients of Hamiltonians and then we check that the hypotheses are still  satisfied after a Birkhoff normal form reduction. 
We propose induction hypotheses as follows:
\begin{align}\label{i0}
&|c_{\bf m}| \le 1+ N_{\al}^{-2s_1} \Pi_j ( |m_j| \wedge N_{\al})^{2s_1}N_{\al}^{2\tau}  =: 1+ N_{\al}^{-2s_1}\PI
   \quad
&\mbox{ for } \Sigma_{2_2}\\ \label{i1} 
&|\cmn| \le N_{\al}^{-4s_1}  \Pi_j ( |m_j| \wedge N_{\al})^{2s_1}N_{\al}^{2\tau} 
\Pi_j (|n_j| \wedge N_{\al})^{s_1}N_{\al}^{\tau} =: N_{\al}^{-4s_1}\PI \Pq  \quad
&\mbox{ for } \Sigma_3,\, \Sigma_5\\ \label{i2}
 &|\cmn| \le   \Pi_j ( |m_j| \wedge N_{\al})^{2s_1}N_{\al}^{2\tau} \Pi_{j\ge 4} (|n_j| \wedge N_{\al})^{s_1}N_{\al}^{\tau} =: \PI \Pqq &\mbox{ for} \Sigma_4\\
 &|\cmn|\le  |\bm||\bn| \Pi_j ( |m_j| \wedge N_{\al})^{2s_1}N_{\al}^{2\tau} 
 \Pi_j (|n_j| \wedge N_{\al})^{s_1}N_{\al}^{\tau} =: |\bm||\bn|\PI \Pq  \quad
&\mbox{ for } \Sigma_6  \label{i3}\\
& |\cmn|\le \Pi_j ( |m_j| \wedge N_{\al})^{2s_1}N_{\al}^{2\tau} 
 \Pi_j (|n_j| \wedge N_{\al})^{s_1}N_{\al}^{\tau} =: \PI \Pq  \quad
&\mbox{ for } \Sigma_7 \label{i4}.
\end{align}
{Note that there is no smallness hypothesis on $ \Sigma_{2_1}$. In fact, eventually, $\Sigma_2$ need not to be small as it is fully resonant term. However, the Poisson bracket with $\Sigma_{2_2}$ produces other terms $ \Sigma_3~\Sigma_7$. Thus, we require the hypothesis on $\Sigma_{2_2}$.}
In $\Sigma_k$,  we put $|n_j|$ into the decreasing order, $|n_1|\ge |n_2| \ge| n_3| \ge \dots$.  We denote
\begin{align*}
\Pi_j ( |m_j| \wedge N_{\al})^{2s_1}N_{\al}^{2\tau} &:=  \PI  \\
\Pi_j (|n_j| \wedge N_{\al})^{s_1}N_{\al}^{\tau} &:= \Pq  \\
\Pi_{j\ge 4} (|n_j| \wedge N_{\al})^{s_1}N_{\al}^{\tau} &:= \Pqq.
\end{align*} 
It follows 
\begin{align}\label{fac}
\I{\bm}{a}\I{\bmp}{a} =\I{ (\bm, \bmp)}{a}, \quad
\Q{\bn}{a}\Q{\bnp}{a} =  \Q{ (\bm, \bmp)}{a}.
\end{align}
One can verify the initial Hamiltonian \eqref{hqn} 
fits into the above description up to 
the initially given constant. For \eqref{hqn}, $N_1=1 $ and $ \cmn= 1 $. Note that 
\[ \sum_{n_1 - n_2 + n_3- n_4} q_{n_1} \bar q_{n_2} q_{n_1} \bar q_{n_4}
=  2 \left(\sum I_m \right )^2 - \sum I_m^2  + 
\sum_{  \{n_1, n_3\} \bigcap \{n_2, n_4\} = \emptyset }q_{n_1} \bar q_{n_2} q_{n_1} \bar q_{n_4}.\] 
Then for the initial Hamiltonian \eqref{hqn}, 
\begin{align*}
 c_{\bf m} = 0 \, \mbox{ for } \Sigma_{2_1}, \Sigma_{2_2} \quad
 c_{\bf m \bf n}  = 1 \, \mbox{  for } \Sigma_3, \Sigma_4,\quad 
 c_{\bf m \bf n} = 0 \,  \mbox{  for } \Sigma_5, \Sigma_6, \Sigma_7
\end{align*}
Now we explain on the form of \eqref{Harrange} and the coefficient bounds 
\eqref{i0} -- \eqref{i4}. 
First of all,  $\cmn$ is naturally bounded by product form: 
\[\cmn \le C \Pi _j(|m_j|\wedge N_a)^{2s_1} \Pi_j(|n|\wedge N_a)^{s_1},\]
then the sum $\sum \cmn \jm \qn $ converges due to  $|q_{n_j}| \le \ep|n_j|^{-s}$ for each $j$. 
To obtain Theorem \ref{main},    $\sum \cmn \jm\qn$ is not only to converge but also be smaller than 
$\ep^{-B}$. For this purpose, we choose large parameters $A, N_\infty$, such that the sum of monomials in 
 $\Sigma_4, \Sigma_5, \Sigma_6,$ and $\Sigma_7$ are small. $\Sigma_3$ may contain harmful terms when the cut-off parameter $N_a$ is small.  But by iteration, we push $N_a$ to larger number and to obtain the smallness from the factor $N_a^{-4s_1}$. To be consistent with this, we impose the condition $  N_a \le |n_-|$. Then  $N_a^{-4s_1} \Pi_{j=1}^4 (|n_j|\wedge N_a)^{s_1}N_a^{\tau} \ge 1$ indeed, so the induction hypothesis \eqref{i2} holds true for \eqref{hqn} 
 for any $N_a$.\\
\indent 
For a given parameter $N_a < N_{a+1}$, we want to remove harmful nonresoant terms of $N_a \le |n_-| \le N_{a+1}$ in $\Sigma_3$ via the Birkhoff normal form transformation. For this purpose we choose Hamiltonian for time 1-shift
\[ F =  -i \sum_{ \substack{3 \\ N_{\al} \le |n_-| < N_{\al+1}}} \frac{\cmn}{\Omega(\bf n)} \jm \qn,\]
then by \eqref{b1} we have  
\[ \{ F, \sum \om_n I_n \} = - \sum_\subal \cmn \jm \qn.\]
\indent
Now, we explain how we proceed normal forms. We set  a  increasing sequence of parameters, 
\[ N_1=1 < N_2 < \cdots <N_{\al} <N_{\al+1} < \cdots < N_{\be} \le N_\infty.  \]
For a fixed $N_a$, in the middle of procedure, Hamiltonians are of the form \eqref{Harrange}. Then we take the Poisson bracket with $F$, $\{ F, \Sigma_k\} $ for each $k=1,\cdots, 7$, and  check 
  the generated  polynomials  in  $\{F, \Sigma_k\}$  can be put into one of $\Sigma_j's$ by
  showing the corresponding  induction hypothesis still holds (see \eqref{bra}). 
  Moreover we show  $H\circ \Phi_F$ allows the decomposition \eqref{Harrange}  satisfying the induction hypothesis with respect to  $N_{\al}$ (Proposition \ref{induction1}). In this step,   $\Sigma_3$ consists of 
  polynomials with a frequency cut-off $ N_{\al+1}$ or that with an extra 
   $\ep$ multiplied.  We iterate the Birkhoff normal forms until all polynomials 
 in   $\Sigma_3$ with $ N_a \le |n_-|  \le N_{a+1} $ are put into $\Sigma_7$.
   Next, we increase the cut-off parameter $N_\al$ to $N_{\al+1} $ and rearrange the Hamiltonian as in $\eqref{Harrange}$ with respect to $N_{\al+1} $ (Propositon \ref{induction3}). 
   We iterate this procedure until $N_a$ reaches  
a sufficiently large $N_b$, for which we will have a desired estimates on coefficients.\\
\indent
In the following we show how to obtain $H_{\al+1}$ from $H_{\al}$ with details.  It will be 
summarized in Proposition \ref{induction3}. First, we study the sums that $\{F, H_{\al}\}$ generates. What it follows $H$ stands for $H_{\al}$, taken off the subscript for notational simplicity.\\

  $\{F,H \} $ gives rise to $ \{F, \sum_k\}$ for $k=1,\dots,7$ and each case results in several types of sum. \\
\\
{\bf(i) $\{F,\sum I_l^2\}$ }\\
  \\
$ \{ F , \sum I_l^2 \} $ is only of $\Sigma_3$ type; 
\[ \{ F , \sum I_l^2 \}  =   \sum_{\subal} \sum_{ n\in \bn}
\frac{\cmn}{\Omn} \sharp n (\bn) I_n \jm\qn= \sum_n \sum_{\subal}\lbn
\frac{\cmn}{\Omn}I_n \jm\qn . \]
We write the sum $\sum_{n\in \bk}\sharp n (\bk) I_n - \sum_{n\in \bp}
\sharp n (\bp) I_n$ as $\sum_{ n\in \bn}\sharp n (\bn) I_n$, and bound them by 
$\sum_n \lbn I_n$. We apply the estimate of $\Omn$ in \eqref{dsize} with noting that $r (= \mbox{degree} of \qn) \le A $, $\mini \le N_{\al+1}$ for 
the monomial in $F$, and obtain
\begin{align}\label{size}
\Omn ^{-1} \le   \ep^2 N_{\infty}^{A^2}N_{\al+1}^m C^{A/m}.
\end{align}
By \eqref{dsize} and \eqref{i1}, we estimate
\begin{align*}
\left| \lbn \frac {\cmn}{\Omn} \right| & \le A\edeno\dec \cobm \cobn.
\end{align*}
By a trivial bound $ \cobm \le I_{ (\bm, \bmp), N_a}$ we have
\begin{align}\label{sum3-1}
\left| \lbn \frac {\cmn}{\Omn} \right|\le  \ep\dec \I{(\bm, n)}{\al}\Q{\bn}{\al} 
\end{align}
under a condition 
\begin{align}\label{cond1}
\ep A \deno \le 1.
\end{align}
\\
{\bf(ii) $\{F,\Sigma_2\}$} \\
\\
Note that $ \{ F , \sum _{2_1} \} = 0   $ due to frequency separations. \\
\noindent
$ {\{ F , \sum _{2_2} \} } $ becomes of $\Sigma_3$ type or $\Sigma_5$ type. Using \eqref{b3}
\[  \{ F , \sum a_\bl I_\bl \} \le A^2  \sum _{\subal}
\sum_{n\in \bn \cap \bl} \frac{\cmn c_{\bl}}{\Omn}  \jm \jl^{\sim n}\qn.\]

If $ n \in \bm \cap \bl $, then $|(\bm , \bl) \setminus \{n\} | = |\bm|+|\bl|-1$. To obtain the coefficient for $\jm \jl^{\sim n}$, we make product for $ m_j \in (\bm , \bl ^{\sim {n}}) $, and denote 
$$     \I{(\bm, \bl^{\sim n})}{\al} =  \Pi_{j= 1}^ {{|\bm| + |\bl|-1}} ( |m_j| \wedge N_{\al})^{2s_1}N_{\al}^{2\tau}   $$  
ii-1) $\Sigma_3$ type\\
We estimate separately the cases of $ | a_\bl| \le 1$ and $| a_\bl| \le N_{\al}^{-2s_1} \cobl$.
If $| a_\bl| \le 1$, we have
\begin{align} \nonumber
 A^2\left|  \frac {\cmn c_{\bl}}{\Omn} \right| &\le A
 ^2
\edeno \dec \cobm  \cobn\\\nonumber
&\le A^2
\edeno\dec \pim{\lbm + \lbl -1}{\al} \cobn\\ \nonumber
& \le \ep \dec\pim{\lbm + \lbl -1}{\al} \cobn \\  \label{sum3-2}
& \le \ep \dec \I{(\bm, \bl^{\sim n})}{\al} \Pq
\end{align}

under a condition
\begin{align}\label{cond20}
A^2\ep \deno \le 1.
\end{align}
On the other hands, if $| c_\bl| \le N_{\al}^{-2s_1} \cobl$, 
 we have an  factor $ (|n| \wedge N_{\al})^{2s_1} = N_{\al}^{2s_1}$ due to the loss of $I_n$, and 
\begin{align}\nonumber
 A^2\left|  \frac {\cmn c_{\bl}}{\Omn} \right| &\le A^2
\edeno \dec N_{\al}^{-2s_1}\cobm  \cobl \cobn\\ \nonumber
&\le A^2\edeno\dec N_{\al}^{-2s_1} N_{\al}^{2s_1+2\tau}
\pim{\lbm + \lbl -1}{{al}} \cobn\\ \nonumber
& \le \ep \dec  \pim{\lbm + \lbl -1}{{\al}} \pin{\lbn}{ {\al}} \\ \label{sum3-3}
& \le \ep \dec \I{(\bm, \bl ^{\sim n})}{\al} \Pq
\end{align}
under a condition
\begin{align}\label{cond2}
A^2\ep \deno N_{\al}^{2\tau} \le 1.
\end{align}
Note that we have an extra $\ep$ in the coefficient $ c_{(\bm, \bl^{\sim n}) \bn} $ when $\{F, \Sigma _2\}$ results in  $\Sigma_3$.\\
ii-2) $ \Sigma_5$ type\\
The estimate and the condition are similar. \\
\\
{\bf(iii) $\{F,\Sigma_3\}$} \\
It  generates one of types of  $\Sigma_2$, $\Sigma_3$,  $\Sigma_4$, or $\Sigma_5$.
\begin{align*}
 \{ F , \Sigma_3\} 
\le &  \sum _{\substack{\bm,\bn, \\ \bm',\bn'}} A^2 \left ( \sum_{ \bm\in \bnp \cap \bn} 
\frac{\cmn \cmnp}{\Omn} \jm^{\sim n} \jmp \qn \qnp  +
 \sum_{ n\in \bmp \cap \bn} 
\frac{\cmn \cmnp}{\Omn} \jm \jmp^{\sim n} \qn \qnp \right.\\
& \left.  \quad + \sum_{ n\in \bn \cap \bnp} \frac{\cmn \cmnp}{\Omn} \jm \jmp \qn^{\sim n} \qnp^{\sim n} \right) \\
& \quad =  S_1  \,  + \, S_2 \, + S_3,
\end{align*}
where $\bn$, $\bn'$ run through $ N_{\al} \le \mini < N_{\al+1}$,
 $ N_{\al} \le |n'_-| < N_{\infty}$ respectively in the sum $\sum_{\bn, \bnp}$.
Let us treat $S_1$ first and then $S_3$,  for $S_2$ is treated similarly.
\\
iii-1) $\Sigma_3$ type of $S_1$
\begin{align*}
& A^2 \left|  \frac {\cmn \cmnp}{\Omn} \right|\\
 &\le A^2
\edeno \dec\dec \cobm \cobmp \cobn \cobnp\\
&\le A^2\edeno \dec\dec N_{\al}^{2 s_1+2\tau}
\pim{\lbm + \lbmp -1}{\al}\qa{\nnp}   \\
&\le \ep \dec \ia{\msn}\qa{\nnp}
\end{align*}
under a condition
\begin{equation}\label{cond3}
A^2\ep \deno  N_{\al}^{-2 s_1+2\tau}\le 1.
\end{equation}
iii-2) $\Sigma_4$ type of $S_1$\\
At least three entries of $(\bn, \bnp)$ are bigger than $N_{\infty}$.
We have
 \begin{align*}
& A^2 \left|  \frac {\cmn \cmnp}{\Omn} \right|\\
& \le 
 A^2\edeno\dec\hdec \pim{\lbm + \lbmp -1}{\al} \pin{\lbn + \lbnp}{\al}\\
& \le  A^2 \edeno N_{\al}^{-3s_1+3\tau} \pim{\lbm + \lbmp -1}{\al} \Qq{\nnp}{\al}\\
&\le  \ep \ia{\msn}\Qq{\nnp}{\al}
\end{align*}
 under a condition
 \begin{equation}\label{cond33}
\ep A^2\deno N_{\al}^{-3s_1+3\tau} \le 1
 \end{equation}

iii-3)
 $\Sigma_3$ case of  $S_3$\\
  We have
\begin{align} \nonumber
&  A^2 \left|  \frac {\cmn \cmnp}{\Omn} \right|\\ \nonumber
& \le  A^2
\edeno\dec\dec N_{\al}^{2(s_1+\tau)} \ia{\mmp}\pin{\lbn + \lbnp-2}{\al}\\ \label{sum3-4}
&\le \ep\dec \ia{\mmp} \qa{\nsn}
\end{align}
under 
\[ \ep  A^2\deno N_{\al}^{-2s_1} \le 1.\]
iii-4)
 $\Sigma_4$ case of $S_3$ \\
 We have
\begin{align*}
&  A^2 \left|  \frac {\cmn \cmnp}{\Omn} \right|\\
&\le  A^2 \edeno N_{\al}^{-8s_1} \cobm \cobmp \cobn \cobnp \\
& \le  A^2 \edeno N_{\al}^{-8s_1} N_{\al}^{5(s_1+ \tau)}\ia{\mmp}
\Pi_{j\ge 4}^{ \lbn+\lbnp -2} ( |n_j|\wedge N_{\al})^{s_1} N_{\al}^{\tau}\\
&\le  \ep \ia{\mmp}\qa{\nsn}
\end{align*}
under 
\[ \ep A^2 \deno N_{\al}^{-3s_1} \le 1.\]
The  case that   $S_1, S_2, S_3$ generate $\Sigma_5$ term are estimated similarly to $\Sigma_3$. We omit the detail here. \\
We postpone the case where $ \{F, \si_3\}$ generates $\si_2$ in the end of 
the part $ {\bf (iv)}$.
\\
\\
{\bf (iv) $\{F, \Sigma_4\} $ } \\
It gives rise to terms of type $\Sigma_3, \Sigma_4$,  or  $\Sigma_5$. \\
\\
iv-1) 
 $\Sigma_3$ type\\
It is obtained from reduction of a pair of $(q_n, \bar q_n)$, which is the third case in \eqref{b4}. \\
Let us estimate the coefficient bound of 
\[ \sum _{n \in \bn \bigcap \bnp}  A^2\frac{\cmn \cmnp}{\Omn}\jm\jmp\qn^{\sim  n}\qnp^{\sim  n}.\]
We have 
\begin{align*}
&  A^2 \left|  \frac{\cmn \cmnp}{\Omn}\right| \\
&\le A^2 \edeno \dec \pim{\lbm+\lbmp}{\al} \cobn \cobnph.
\end{align*}
Note that for $\{F, \Sigma_4\} $ to be $\Sigma_3$, the reduced $n$ is $|n|\ge N_{\infty}$, 
obviously $ |n| \wedge N_{\al} = N_{\al}$.
And at least two $n'_j$ among $\bar n'$ are $n'_j \ge N_{\infty}$ since 
$\jmp \qnp$ consists of $\Sigma_4$.  The right hand term is bounded by
\begin{align}\nonumber
& \le  A^2 \edeno \dec N_{\al}^{2(s_1 +\tau)} \I{\mmp}{\al} \pin{|\bn| -1}{\al}\Pi_{j= 4} ^{|\bn'|-1}
(\lbnp \wedge N_{\al})^{s_1}N_{\al}^{\tau}\\ \nonumber
&  \le  A^2\edeno \dec \I{\mmp}{\al} \pin{|\bn| -1}{\al}\Pi_{j=1} ^{|\bn'|-1}
(\lbnp \wedge N_{\al})^{s_1}N_{\al}^{\tau}\\ \nonumber
& \le A^2 \edeno \dec  \ia{\mmp} \pin{\lbn + \lbnp-2}{\al}\\ \label{sum3-5}
&\le  \ep \dec \I{\mmp}{\al}\Q{\nsn}{\al}
\end{align}
under 
$ \ep  A^2\deno \le 1$.\\
\\
iv-2) 
 $\Sigma_4$ type\\
   We treat the fisrt and second reducton cases
in \eqref{b4} and the third one separately.
\\
Let us estimate the coefficient bound of 
\[ \sum _{n \in \bn \bigcap \bnp}  A^2\frac{\cmn \cmnp}{\Omn}\jm^{\sim n}\jmp\qn\qnp.\]
We have 
\begin{align*}
&  A^2\left|  \frac{\cmn \cmnp}{\Omn}\right| \\
&\le  A^2\edeno \dec N_{\al}^{2s_1} \pim{\lbm+\lbmp -1}{\al}  \cobn \cobnph.
\end{align*}
If all the  three biggest index  among $ \{ \bn, \bnp\}$ arise in $\bnp$, we bound the right hand side by
\[  A^2 \edeno\dec N_{\al}^{2s_1}  \pim{\lbm+\lbmp -1}{\al}\Pi_{j= 4} ^{\lbn+ \lbnp}
(\lbn \wedge N_{\al})^{s_1}N_{\al}^{\tau}.\]
If  some of three biggest arise in $\bn$, note that  $|n_j| \ge N_{\al}$ for $n_j \in \bn$
and newly included $n'_j$ is such that $|n'_j| \ge  N_{\infty}$. Hence the extra coefficent are cancelled out
in this step. We have 
\begin{align*}
 A^2 \left|  \frac{\cmn \cmnp}{\Omn}\right| \le \ep \I{\msn}{\al} \Qq{\nnp}{\al}
\end{align*}
under 
$ \ep A^2\deno N_{\al}^{-2s_1} \le 1$.
Next we estimate the coefficient bound of 
\[ \sum _{n \in \bn \bigcap \bnp}  A^2\frac{\cmn \cmnp}{\Omn}\jm\jmp\qn^{\sim  n}\qnp^{\sim  n}.\]
\begin{align*}
&  A^2 \left|  \frac{\cmn \cmnp}{\Omn}\right|\\
& \le  A^2\edeno \dec N_{\al}^{2(s_1 + \tau)} \ia{\mmp} \pin{\lbn -1}{\al} 
\Pi_{j= 4} ^{ \lbnp-1}
(\lbn \wedge N_{\al})^{s_1}N_{\al}^{\tau}\\
&\le  A^2\edeno \dec N_{\al}^{2(s_1 + \tau)} \ia{\mmp}
\Pi_{j= 4} ^{\lbn+ \lbnp-2}
(\lbn \wedge N_{\al})^{s_1}N_{\al}^{\tau}
\end{align*}
because the third biggest index among $ \bn \cup \bn' \setminus \{n\}$
is bigger than that among $ \{\bnp\} / \{ n\}$.
We have 
\[   A^2\left|  \frac{\cmn \cmnp}{\Omn}\right| \le \ep  \I{\mmp}{\al}\Qq{\nsn}{\al}  \]
under the condition $\ep A^2\deno N_{\al} ^{-2s_1 + 2\tau} \le 1 $.\\
\\
iv-3) 
The case that $\{F, \Sigma_4\}$ generate $\Sigma_5$ term are estimated as same as $\Sigma_3$.\\
\indent 
$\{F, \Sigma_3\}$ and $\{F, \Sigma_4\}$ can generate $\si_{2_2}$ terms when 
$\overline\qn =  \qnp$. In the induction hypotheses \eqref{i0}-\eqref{i4} we see that the coefficient's bounds for $\si_3$ and $\si_4$ are assumed to be smaller than for $\si_{2_2}$; For 
$\si_3$ it is obvious, and for $\si_4$ it is from
$ \PI \Pqq \le N_{\al}^{-3(s_1 +\tau)} \PI\Pq.$ So by estimates in $ {\bf (iii)}$ and  $ {\bf (iv)}$, 
we have the coefficients of the generated $\Sigma_{2_2}$ term bounded by
\[ |\mathbf{c_m}| \le \ep N_{\al}^{-2s_1}\PI.\]
\\
\\
{\bf (v)} Similarly we have  $\{ F , \Sigma_5 \} = \ep \Sigma_5 + \ep \Sigma_6$, $\{F, \Sigma_6\} = \ep\Sigma_6$,
and  $\{F, \Sigma_7\} = \ep \Sigma_7$. 
Let us verify $\{ F, \Sigma_6\} = \ep\Sigma_6$.  According to \eqref{b4},  we have to show 
\[ A | \bm'| \frac{|\cmn\cmnp|}{\Om(\bn)} \le \ep (\lbm + \lbmp)(\lbn +\lbnp)
I _{((\bm, \bmp)^{\sim n}, N_a)} Q_{((\bn, \bnp), N_a)}\]
when a loss of $I_{\bn}$ occurs, and 
\[ A \lbnp \frac{|\cmn\cmnp|}{\Om(\bn)} \le \ep I _{((\bm, \bmp), N_a)} Q_{((\bn^{\sim n}, \bnp^{\sim n}), N_a)}\]
when a pair of $(q_n, \bar q_n)$ is contracted. For the former, we have 
\begin{align*}
A \lbnp \frac{|\cmn\cmnp|}{\Om(\bn)} &\le 
A \lbmp \edeno N_a^{-4s_1}I_{(\bm, N_a)}I_{(\bm, N_a)}Q_{(\bn, N_a)}Q_{(\bnp, N_a)}\\
&\le A \lbmp \edeno N_a^{-2s_1 + 2\tau}I _{(\msn, N_a)} Q_{((\bn, \bnp), N_a)}\\
&\le \ep (\lbm + \lbmp)(\lbn +\lbnp)
I _{(\msn, N_a)} Q_{((\bn, \bnp), N_a)}
\end{align*}
under the condition  $ A \ep\deno N_a^{-2s_1 + 2\tau} \le 1$.
The other cases are similar. We omit details.\\
\indent
Overall the conditions for  $A, s, \ep, \{N_{\al}, N_{\infty}\}$ is reduced to \eqref{stau} and  
{
\begin{equation}\label{overallcondition}
\ep A^2 N_{\infty}^{4A^2} N_{\infty}^m C^{A/m} \le 1. \end{equation} }
We assume  that $A, s, \ep, \{N_{\al}\}$ satisfy
\begin{align}\label{condition}
 \tau = \frac{10s}{A},  N_{\infty} = \ep^{- \frac{A}{10^2s}},
 s > A^3, \, \mbox{ and } \ep^{\frac 12} C^{A/m} \le 1.
\end{align}

So far, we have proven
\begin{align}\begin{aligned}\label{bra}
\{ F, \si_0\} & = -\Sigma_3\Big|_{N_{\al} \le \mini < N_{\al+1}}\\
\{ F, {\Sigma_1}\} & = \ep\Sigma_3,\\
\{ F, \Sigma_2\} & = \ep \Sigma_3,\\
\{ F, \Sigma_3\} & = \ep \left(\si_2+   \Sigma_3  + \Sigma_4 +\Sigma_5 \right), \\
\{ F, \Sigma_4\} & = \ep( \si_2+ \Sigma_3 +\Sigma_4+\Sigma_5),\\
\{ F, \Sigma_5\} & = \ep  \Sigma_5 + \ep \Sigma_6,\\
\{ F, \Sigma_6\} & = \ep  \Sigma_6,\\
\{ F, \Sigma_7\} & =    \ep\Sigma_7.
\end{aligned}\end{align}
The point is that we have the extra $\ep$ in front of $\si_3$  in $\si_i$ when $i \ge 1
$ (The corresponding estimates are \eqref{sum3-1}, \eqref{sum3-2}, \eqref{sum3-3}, \eqref{sum3-4}, 
and \eqref{sum3-5}).
Let us define 
\[H_F : = H\circ \Phi_F.\]
Recalling \eqref{recall},  the Taylor series expansion formula of $H\circ \Phi_F$ centered at $t=0$, we obtain
\bq\label{taylor}
H_F = \sum_{k=0}^{l} \frac{1}{k!}\{ F,H\}^{(k)} + 
\frac{1}{l!} \int_0^1 (1-t)^l \{F, H\}^{l+1}\circ \Phi_F^t dt.
\eq
We denote 
\[ \{F, H\}^{(k)}:= \{ \underbrace{F,\cdots \{F}_{\mbox{ k times }}, H\},  \cdots\}\]
 and  $\{F, H\}^{(0)} = H$. Under the initial condition $\| q(0)\|_{H^s} \le \ep$ and a
 consistencty condition to be proved in Section $4$, the remainder converges, so we simply write 
 \[ H_F = \sum_{k=0}^{\infty} \frac{1}{k!}\{ H,F\}^{(k)}.\]

\begin{proposition}\label{induction1}
By the induction argument we have
\begin{align*}
  H_{F} =  \si_{0} +  {\Sigma_1} +  (1 + 5\ep)(\Sigma_2 +
   \Sigma_4 + \Sigma_5 + \Sigma_6+ \Sigma_7)
     + \sum_{\subalty}  + 6\ep\Sigma_3.\end{align*}
\end{proposition}
\begin{proof} $ $ \\
By \eqref{bra}, note that $ \{F,\Sigma_0\} $ cancels $\Sigma_3 $ with $N_{\al} \le \mini < N_{\al+1} $, and we have 
\begin{align}\label{1th}
H + \{ F, H\} &= \si_{0} + {\Sigma_1} + \Sigma_3\Big|_{N_{\al+1} \le \mini < N_{\infty}}  + 4\ep \Sigma_3 \\ \nonumber
& + 
(1 +2\ep) 
(\Sigma_2 +   \Sigma_4 +\si_6) +  (1 +3\ep) \Sigma_5 + (1 +\ep)\Sigma_7.
\end{align}
For $k\ge 2$ we assume the induction hypothesis:
\[ \frac{\{F, H\}^{(k)}}{4^k} = 
  \ep^{k-1}\sum_{2, 3, 4, 5} \, + \,  \ep^k\sum_{2, 3, 4, 5,  6, 7} \]
  where we denote $\sum_i + \sum_j $ by $\sum_{i, j}$ for simplicity.
 If $k=2$, it is straightforward that 
 \[ \frac{ \{F, \{F, H\}\}}{4^2} = 
 \ep\sum_{2, 3, 4, 5 } + \ep^2 \sum_{2, 3, 4, 5, 6, 7} \]
 due to \eqref{bra}. Similarly, the induction hypothesis holds for $ k+1$th step;
\[  \frac{\{F, H\}^{(k+1)}}{4^{k+1}}= \frac 13 \{ F, \frac{\{F, H\}^{(k)}}{4^{k}}\}=
  \ep^{k}\sum_{2, 3, 4, 5 } + \ep^{k+1}\sum_{2, 3, 4, 5, 6, 7}.\]
 It holds that 
 \begin{equation}\label{2th}
 \sum_{k\ge 2} \frac{\{F, H\}^{(k)}}{k!} = 
\frac {32}{3}  \ep\sum_{2, 3, 4 , 5 } + \ep^2 \sum_{2, 3, 4, 5, 6, 7} 
 \end{equation}
with  $\frac{4^k}{k!} \le \frac {32}{3}$. The propostion follows by adding \eqref{1th} and 
\eqref{2th}.
\end{proof}
So far we removed monomials of 
$ \nal$ in $\Sigma_3$ and obtain
an extra $\ep$ factor in front of $\Sigma_3$. We will go on until 
the increasing exponent is begger than  $A$ so that we have $\ep^A \Sigma_3$, which joins $\Sigma_6$;
let us consider the normal form transform  of $H_F$ 
by the associated Hamiltonian 
$\ep F$, where we use the same notation $F$ to denote
\[ F = 
 \Sigma_3\Big|_{N_{a+1}\le |n_-|< N_{a+1}}  \frac{\cmn}{\Omn} \jm\qn\]
with $\cmn \jm\qn$ the monomial of $\Sigma_3$ of $H_F$. Here, $\Sigma_3\Big|_{N_{a+1}\le |n_-|< N_{a+1}}$ means the summation of term with condition $N_{a+1}\le |n_-|< N_{a+1}$.
Similarly as Propostion 
\ref{induction1} we compute $H_F\circ \Phi_{\ep F}$ as follows. 
 \indent 
What it follows,  for notational simplicity, we use 
  Proposition \ref{induction1} in the form of
 \begin{align}\label{new1}
   H_{F} =  \si_{0} +  {\Sigma_1} + \Sigma_2 +    \Sigma_3\Big|_{N_{a+1}\le |n_-|< N_\infty}   + \ep\Sigma_3+
   \Sigma_4 + \Sigma_5 + \Sigma_6+ \Sigma_7.
 \end{align}
\begin{proposition}\label{induction2}
By the induction argument we have
\begin{equation*}
  H_{F} \circ \Phi_{\ep F}= \si_{0} +  {\Sigma_1} + \Sigma_2 +
 \Sigma_3\Big|_{N_{a+1}\le |n_-|< N_\infty} 
 + \ep^{2}\Sigma_3 + \Sigma_4 + \Sigma_5 + \Sigma_6+ \Sigma_7
\end{equation*}
\end{proposition}
\begin{proof}
First we note that 
\[ \{\ep F, \si_0\} = -\ep \Sigma_3\Big|_{N_{a+1}\le |n_-|< N_{a+1}}  .\]
Hence
\[ H_F +  \{\ep F, \si_0\} = \sum_{0, 1, 2} + (1+\ep) \Sigma_3\Big|_{N_{a+1}\le |n_-|< N_\infty}  + \sum_{4,5,6, 7}.\]
On the other hand, we have 
\begin{align*}
\{ \ep F, H_F -\si_0\} & = \{ \ep F, \sum_{1,2} + (1+\ep)\Sigma_3 + \sum_{4,5,6, 7}\} \\
& = \ep (\ep^2 +3 \ep) \Sigma_3 + \ep ( \ep^2 + 2\ep) \Sigma_4 + \ep (\ep^2 +4\ep) \Sigma_5  +\ep^2 \Sigma_6 +\ep\Sigma_7\\
& = \ep^2 \Sigma_3 + \ep^2 \Sigma_4 + \ep^2 \Sigma_5 +\ep^2 \Sigma_6
\end{align*}
according to the policy in front of Proposition \ref{induction1}. We have 
\begin{align*}
H_F + \{\ep F, H_F\} = 
\sum_{0, 1,2} + (1+\ep) \Sigma_3\Big|_{N_{a+1}\le |n_-|< N_\infty}  + \ep^2 \Sigma_3 + (1 + \ep^2) \sum_{4,5,6} +\Sigma_7
\end{align*}
 Note that we have
\[ \{\ep F, H_F\} = (1+\ep^2) \Sigma_3 + \ep^2\sum_{4,5,6}+\ep \Sigma_7.\]
Assume the induction hypothesis hold for $k\ge 2$
\[ \frac{1}{k!} \{\ep F, H_F\}^{(k)} = \ep^{2k} \sum_{ 3,4,5,6} + \ep^{2k-1}\Sigma_7.\]
The $k=2$ case is established by same computations as above. Then it is straightforward that the  $k+1$-step 
holds: 
\begin{align*}
\frac{1}{(k+1)!} \{ \ep F, H_F\}^{(k+1)} 
& = \frac{1}{k+1} \{ \ep F,   \frac{1}{k!} \{\ep F, H_F\}^{(k)} \}\\
&=\frac{ 2 \ep^{2(k+1)}}{k+1} (\Sigma_3 + \Sigma_4 +\Sigma_5 +\Sigma_6)+\frac{\ep^{2k}}{k+1}\Sigma_7.
\end{align*}
So  we have
\begin{align*}
&H_F \circ \Phi_{\ep F}  = \sum_{k=0}^{\infty} \frac{1}{k!} \{ \ep F, H_F\} \\
& = \sum_{0, 1,2}\, + (1+ \ep) \Sigma_3\Big|_{N_{a+1}\le |n_-|< N_\infty}  + \ep^2 \Sigma_3 + (1 +\ep^2) \sum_{4,5,6}
+  \left( \si_{k\ge 2}  \ep^{2k} \right)\sum_{3,4,5,6} + \left(  \si_{k\ge 2} \ep^{2k-1} \right) \Sigma_7\\
& =  \si_{0} +  {\Sigma_1} + \Sigma_2 +
 \Sigma_3\Big|_{N_{a+1}\le |n_-|< N_\infty} 
 + \ep^{2}\Sigma_3 + \Sigma_4 + \Sigma_5 + \Sigma_6+\Sigma_7
\end{align*}
as desired.
 \end{proof}
 We can repeat the above procedure all over again by taking the normal form transformation
 with $\Phi_{\ep^2 F}$.  Denote 
 \[ H_F \circ \Phi_{\ep F}\circ \Phi_{\ep^2 F}\cdots 
 \circ \Phi_{\ep^{k}F} : = H_{F^{(k)}}.\]
 Inductively, we have the following proposition.
 \begin{proposition}\label{induction3}
 If $k > A$ we have
 \[H_{F^{(k)}} = \si_{0} +{\Sigma_1} +\Sigma_2 + 
 \Sigma_3\Big|_{N_{a+1}\le |n_-|< N_\infty} 
 +\Sigma_4 + \Sigma_5 + \Sigma_6 +\Sigma_7.\]
 For such $k$ we denote $H_{F^{(k)}}$ by $H_{\al+1}$. 
The coefficients for  $H_{\al+1}$ satisfy the induction hypothesis in \eqref{i0} -- \eqref{i3} replacing $N_{\al}$ by $N_{\al+1}$.
  \end{proposition}
  \begin{proof} $ $\\
We show only the second assertion. At the time of  reaching $k > A$
  the coefficients for $H_{\al+1}$ remain to be bounded as  \eqref{i0} -- \eqref{i3}.
  To upgrade $\al$ to $\al+1$ we check \eqref{i0} -- \eqref{i3} separately as follows.\\
   The monomials of $\sum_{2_2}$ satisfy $|n_+| \ge N_{\al+1}$, while the other monomials goes to $\sum_{2_1}$
for which we impose no condition.   So there is at least one $m_j$ of 
  $|m_j| \ge N_{\al+1}$, and for this $m_j$ we have
  \[ |m_j| \wedge N_{\al}  \le  \frac{N_{\al+1}}{N_\al} |m_j| \wedge N_{\al+1},\]
  \begin{align*}
 \hdec\cobm &\le N_{\al+1}^{-2s_1}\left( \frac{N_{\al+1}}{N_\al} \right)^{2s_1} \cobm \\
 & \le N_{\al+1}^{-2s_1}\left( \frac{N_{\al+1}}{N_\al} \right)^{2s_1}\left( \frac{N_{\al+1}}{N_\al} \right)^{-2s_1}  \cobmu.
  \end{align*}
  The monomials of  $\Sigma_3$ satisfy $N_{\al+1} \le |\bn_-|$ and the degree of $\qn $ is at least $4$. 
  We have
  \begin{align*}
  &\dec\cobm\cobn \\ &\le   N_{\al+1}^{-4s_1}\left( \frac{N_{\al+1}}{N_\al} \right)^{4s_1} 
  \left( \frac{N_{\al+1}}{N_\al} \right)^{-4s_1}\cobmu\cobnu.
  \end{align*}
  The hypothesis \eqref{i2} for $\Sigma_4$ is automatically upgraded. The monomials of  $\Sigma_5$ are of degree  bigger than $A$. We have 
 \begin{align*}
  &\dec\cobm\cobn \\
  &\le   N_{\al+1}^{-4s_1} \left( \frac{N_{\al+1}}{N_\al} \right)^{4s_1}  \left( \frac{N_{\al+1}}{N_\al} \right)^{-A\tau}\cobmu\cobnu.
    \end{align*}
	Since $\tau = \frac{10 s}{A}$, the extra coefficient is smaller than $1$.\\
Lastly  $\Sigma_6$,  $\Sigma_7$ are automatically upgraded.


  \end{proof}
We set $N_1 =1$ and $N_{2} = N_{\infty}$, and perform Proposition \ref{induction1}
to Proposition \ref{induction3}. Then $\Sigma_3$ is empty. (The emptiness of $\Sigma_3$
is not crucial for the following analysis). The final Hamiltonian is written as 
\begin{align}\begin{aligned}\label{final}
H_{\be} &
 = \si_{0} + {\Sigma_1} +\Sigma_2 \mbox{ : fully resonant terms  } \\
&+ \sum_{A < \text{deg} \le 2A}  a_{\bm \bn} \jm\qn + \sum_{\mu(\bn)> N_{\infty}} b_{\bm \bn} \jm \qn 
+ \sum_{ \text{deg} > 2A} c_{\bm \bn} \jm\qn \\
& + \ep^A \sum_{\text{deg} <A} \dmn \jm\qn,
\end{aligned} \end{align}
where 
\begin{align}\label{finala}
|a_{\bm \bn}| & \le  N_{\be}^{-4s_1}  \cobmb \cobnb,\\
\label{finalb}
|b_{\bm \bn}| & \le  \cobmb \cobnhb, \\ 
\label{finalc}
| c_{\bm\bn} |&\le   \lbm\lbn \cobmb \cobnb,\\
\label{finald}
|\dmn| & \le \cobmb \cobnb.
\end{align}
The bound \eqref{finala} correspond to the sums, $\Sigma_3, \Sigma_5$, \eqref{finalb} to $\Sigma_4$, 
 \eqref{finalc} to the sum $\Sigma_6$, and \eqref{finald} to $\Sigma_7$.

\section{The proof of Theorem \ref{main}}

\subsection{Estimates on the symplectic transformations} $ $ \\
Before proceeding to prove Theorem \ref{main}, we mension that the stability condition is preserved under the symplectic transforms. Indeed, the new Hamiltonian flow $\widetilde{q}_n $ is obtained from a time-$1$ shift for the evolution
\[ i \dot{q_n} = \frac{\pa F}{  \pa \bar q_n}, \quad q_n(0) = q_n, \quad q_n(1) = \widetilde{q}_n.\] 
Using the definition of $F$ we estimate that $ q_n \sim \widetilde{q}_n $ as follows:
\begin{align*}
&\sum |n|^{2s} | q_n(1) - q_n(0)|^2 \le \sum |n|^{2s} \int_0^ 1 | \text{Im}  \, \bar q_n \pa F / \pa \bar q_n| \\
&\le  \edeno\sum |n|^{2s} \sum _{n\in \bn} |\cmn| |\jm||\qn| \\
&\le  \edeno  \sum_{\substack{ l(\bn)=0\\ 4\le |\bn| \le A} }|\bn|^s |n_1|^s |n_2|^s |q_{n_1}| |q_{n_2}|
 \cobm |I_{m_j}|\Pi_{j\ge 3} (|n_j|\wedge N_{\al})^{s_1} N_{\be}^{\tau} |q_{n_j}|\\
 & \le  \edeno \ep^2\sum_{r=2}^{A-2} r^s \sum_{|\bn| =r} \Pi _{j\ge 3}^{r} (|n_j|^{-(s-s_1) }\ep) \\
&\le \edeno \sum _{r=2}^{A-2} r^s \ep^r \\
& \le \edeno A^s \ep^2 \le \ep^2
\end{align*}
under a condition
\begin{align}\label{ep-s}
A^s \ep  \le 1,
\end{align} and \eqref{overallcondition}.

\subsection{The proof of Theorem \ref{main}} $ $ \\
Fix $A = 200B$. Assume that all parameters satisfy size conditions given so far. (See Appendix \ref{parameter} for summary.) As we have only $N_b=N_2= N_\infty$, we denote it by $N$. We prove that the Hamiltonian flow $ \{ q_n(t)\}$ generated by $H_{\be}$ remains
\[ \| q(t)\|_{H^s} \le 2\ep  \quad \mbox {for} \quad t \le \ep^{-B}.\]
The flow $\{q_n(t)\} $ is given by 
\[ iq_n = \frac{\pa H_{\be}}{\pa \bar q_n}, \]
hence 
\begin{equation}\label{Im}
\left|  \ddt |q_n|^2 \right| \le  \sum ( |\amn| + |\bmn|+|\cmn|+ |\dmn|) 
\left| \mbox{Im} \left ( \frac{ \pa (I_\bm q_\bn)}{\pa  \bar q_n} \bar q_n \right) \right|, 
\end{equation}
where $\amn, \bmn, \cmn$ are bounded by \eqref{finala} -- \eqref{finald}, and $\cmn =0$  if the degree of 
$I_{\bm}q_{\bn}$ is less than $2A$.
The consistency assumption for $ \{q_n\}$ is 
\bq \label{consi} \| q\|_{H^s} \le \ep.\eq
By  \eqref{Im}, we have
\begin{align}\nonumber \label{time} 
\sum_n& |n|^{2s} || q_n(t)|^2 - | q_n(0)|^2 | \le \\
\int_0^T 
\sum_n& |n|^{2s} \sum _{\substack{\bm, \bn \\ l(\bn) =0, n\in \bn}} 
( |\amn| + |\bmn| + |\cmn| + |\dmn|) |q_n|
 \left| \frac{ \pa (I_\bm q_\bn)}{\pa  \bar q_n}  \right|. 
 \end{align}
 We want to bound the integrand of \eqref{time} by $\ep^B$.
 Note that the sums is taken over $\displaystyle\sum _{\substack{\bm, \bn \\ l(\bn) =0}}$. 
 If $\bn = (n_1, \dots, n_l)$, we bound $|n_1^*| \le l|n_2^*|$ from 
 the relation $n_1 - n_2 + n_3 \cdots = 0$. Moreover, we use the estimates for $|n|^{2s}$,
 \begin{align*}
 &|\bn| \le A:  \quad  |n|^{2s} \le A^s |n_1^*|^s|n_2^*|^s,\\
 & |\bn|= l >A : \quad  |n|^{2s} \le l^s |n_1^*|^s|n_2^*|^s \le  A^s N^{ls/A} |n_1^*|^s|n_2^*|^s
 = A^s N^{l \tau /10}|n_1^*|^s|n_2^*|^s,
 \end{align*}
 using $l/A < N^{l/A}$. Then, we estimate
 \begin{align*}
 & \sum_n\sum _{\substack{\bm, \bn \\ l(\bn) =0, n\in \bn}} |n|^{2s}|\amn| |\jm| |\qn| \\
 &\le  N^{-4s_1}A^s \sum _{\substack{\bm, \bn \\ l(\bn) =0}}
 |n_1^*|^s|n_2^*|^s \cobmb \Pi_j (|n_j|\wedge N)^{s_1} N_{\be}^{\frac{11}{10}\tau}|\jm| |\qn|\\
 & \le  N^{-4s_1}A^s N^{2(s_1+ 2\tau)}
 \sum _{\substack{\bm, \bn \\ l(\bn) =0}} |n_1^*|^s|n_2^*|^s
 |q_{n_1^*}| |q_{n_2^*}|
 \Pi_j (|m_j|\wedge N)^{2s_1}N^{2\tau}I_{m_j} \Pi_{j\ge 3}(|n_j|\wedge N)^{s_1} N^{2\tau}|q_{n_j}| \\
 & \le N^{-s_1}A^s \| q\|_{H^s}^2 (\sum_{r\ge 2}r\ep^r)^2\\
& \le  N^{-s_1}A^s \ep^3\le \ep^{A/100} 
 \end{align*}
by \eqref{ep-s},
 where we use $|q_{n_j}|  \le \ep|n_j|^{-s}$ and 
 $N^{2\tau} \sim \ep^{-\frac{1}{100}}$ to bound 
 \bq\label{integrable}(|m_j|\wedge N)^{2s_1}N^{2\tau}I_{m_j} \le \ep^{\frac{199}{100}}|m_j|^{-2(s-s_1)} =
  \ep^{\frac{199}{100}} |m_j|^{-20 s/A}, \eq
 and in turn
 \[  \sum_\bm \Pi_j ( |m_j|\wedge N)^{2s_1}N^{2\tau} I_{m_j} 
 \le \sum_{r\ge 2} \ep^{\frac {199}{100}r} \sum_{|\bm| =r} \Pi_{j=1}^r |m_j|^{-20s/A} 
 \le \sum _{r\ge 2} \ep^{r}. \]
 In the sum $\Sigma_4$,the degree of $\qn$ is less than $A$ and $\mu(\bn) > N$. We have
 \begin{align}\label{bmn-part}\begin{aligned}
 &\sum _{\substack{\bm, \bn \\ l(\bn) =0}} |n|^{2s}|\bmn| |\jm| |\qn| \\
&\le A^s\sum _{\substack{\bm, \bn \\ l(\bn) =0}} 
 |n_1^*|^s|n_2^*|^s |q_{n_1^*}||q_{n_2^*}| |q_{n_3^*}|
\Pi_j (|m_j|\wedge N)^{2s_1} N^{2\tau}I_{m_j}\Pi_{j\ge 4} (|n_j|\wedge N)^{s_1} N^{\tau}|q_{n_j}|\\
 &\le A^s N^{-s} \| q\|_{H^s}^2 \ep^2 \le \ep^{A/200}.
 \end{aligned}\end{align}
 Now we estimate the $\cmn$ part. The needed  $\ep^ {A/100}$ factor is obviously from that $\Sigma_6$ consists of  $I_{\bm}q_{\bn}$ with the degree bigger than $2A$.  Let us assume $\lbm > A/2$. The other case $\lbn > A$ can be treated same.
  \begin{align*}
 &\sum_n\sum _{\substack{\bm, \bn , \lbm >A /2\\ l(\bn) =0, n\in \bn}} |n|^{2s}|\cmn| |I_\bm| |\qn| \\
 &\le \sum_n \sum _{\substack{\bm, \bn , \lbm >A/2 \\ l(\bn) =0, n\in \bn}} |n|^{2s}
 \lbm \lbn \cobmb \Pi_j ( |n_j|\wedge N_b)^{s_1} N_b^{\tau} |I_\bm| |\qn|\\
 &\le \sum _{\substack{\bm, \bn , \lbm >A/2 \\ l(\bn) =0, n\in \bn}}  |n_1^*|^s|n_2^*|^s \lbm \lbn 
 |q_{n_1^*}| |q_{n_2^*}|
 \Pi_j (|m_j|\wedge N)^{2s_1}N^{2\tau}I_{m_j} \Pi_{j\ge 3}(|n_j|\wedge N)^{s_1} N^{2\tau}|q_{n_j}| \\
 & \le  \|q\|_{H^s}^2  \sum_{\lbm >A/2}\lbm \cobmb    \sum _ {\lbn \ge 4}  \lbn  \Pi_{j\ge 3} ^{\lbn}(|n_j|\wedge N)^{s_1} N_{\be}^{2\tau}|q_{n_j}|  \\
 &\le  \|q\|_{H^s}^2 \sum_{r\ge A/2+1} r \Pi_{j=1}^r \sum_{m_j } \ep^{\frac{199}{100}} |m_j|^{- 20s/A}
 \sum_{r\ge 4} r \Pi_{j=3}^r \sum_{n_j} \ep^{\frac {99}{100}} | n_j|^{-20s/A}\\
 & \le \ep^2  \sum_{r\ge A+1}r\ep^{\frac {199}{100}r} \sum_{r\ge 4} r \ep^{\frac{99}{100}r} \le \ep^{A/100}
 \end{align*}
by the consistent assumption  \eqref{consi} and the relation \eqref{integrable}.\\
\indent
The contribution of $|\dmn|$ can be similarly bounded since $\ep^A < N^{-4s_1}$.
\qed

\appendix

\section{ Summary of parameters}\label{parameter} $ $ \\
We introduce many parameters in the analysis. Here, for reader's convenience, we summarize the size relation of parameters. 
In the Theorem \ref{main}, parameter $B>0$ is given. Then we consecutively choose $A, s, \tau, N_{\infty}$ as follows:  
\begin{align*}
& A > 200B, \quad  s > A^3\\
&s =s_1 + 5\tau, \quad \tau = \frac{10 s}{A},  \\
&  N_{\infty} = \ep^{- \frac{A}{10^2s}},
\end{align*}
and further $m, s, A$ are such that 
\[ m \le \frac s A.\]
Also we choose $\ep = \ep(C, A, s)$ such that
\[\ep A^s \le 1, \quad  \ep^{\frac 12} C^{A/m} \le 1, \]
Choosing $ A=200B, \, s=A^3, \, m=A^2$, we have $ A^2 \ll \tau=50A^2 \ll s_1. $ 
In fact, the conditions imposed in Section 3 are reduced to 
\[ \ep A^2 N_{\infty}^{A^2} N_{\infty}^m C^{A/m} \le 1,\]
and it can easily be verified from the above choices.
{Note that in this work, we need only $ N_1, N_\infty$.}

\section{Proof of Lemma \ref{lblemma}}\label{AppB}
\noindent
Lemma \ref{lblemma} is the direct consequence of Proposition 4.
We closely follow the argument in \cite{FG}; the difference here is that the exponent of $\langle \bn_+ \rangle$ involves only $r$. In the original version in \cite{FG} it comes $ \langle \bn_+ \rangle ^{-4rm}$
in Proposition 4.

\begin{lemma}\label{Lem}
Fix $\ga$ and $m > d/2$. There exist a set $F_{\ga}' \subset \mathcal W$
whose mearuse is larger than $1-\ga$ such that if 
$V \in F_{\ga}'$ then for any r 
\[ |\Om(\bn) - b | \ge \ga \frac{\ep_1}{10}
 \langle \bn_+ \rangle^{-3(r+1)} \langle \bn_-\rangle^{-m}\]
for any non resonant $\bn = (n_1, \cdots, n_{2r})$ and for any $b\in \mathbb Z$.
\end{lemma}
Let us think the case in which we have two independent random variable $x, y$,
uniformly  distributed over $[- M,  M]$. 
Define a set $A \subset \bbr^2$ by
$ I = \{(x, y)| |x- y + c| \le \eta \} $ for $c \in \bbr$
then 
\[ \mathbb P [ (x, y) \in A ] = |A| = 
\int _{-M}^{M} \int_ { [ -\eta + y -c, \eta+y-c] \cap [-M, M]} dxdy \le 4M\eta.\]
For $A$ not to be empty, $c$ satisfies $|c| \le 2M + \eta$.
Similarly, let $x_1, \dots, x_n $ be indepenent random variables, each 
uniformly  distributed over $[- M,  M]$, and 
\[ A = \{ | a_1x_1 + a_2x_2 + \cdots + a_n x_n + c | \le \eta \}.\]
Assuming $|a_n|= \max\{|a_1|, \dots, |a_n|\}$, we have
\begin{align}\label{nrandom} 
 \mathbb P [(x_1, \dots, x_n) \in B ] \le
\int_{ [-M, M]^{n-1}} \int _{ \frac {1}{a_n}( -\eta -d- \sum_{i=1}^{n-1}a_i x_i)}^
{ \frac {1}{a_n}( \eta -d- \sum_{i=1}^{n-1}a_i x_i)}\chi_{[-M, M]}dx_n \cdots d x_1
\le (2M)^{n-1}\frac{2\eta}{|a_n|}.
\end{align}
For $A$ not to be empty, 
$d$ satisfies $ |d| \le M(|a_1| + \cdots |a_n|)  +\eta$.\\
\begin{proof}
Let $\bn = (n_1, \dots, n_{2r}) = (\bk, \bp)$ and $b \in \mathbb Z$ be given and 
$\eta(\bn)$ be chosen later.
By  \eqref{nrandom}
we have
\begin{align*}
 \mathbb P_{\bn, b} [ \{ (\sigma_{\bk}, \sigma_{\bp})\in [- 1/2 ,  1/2]^{2r}|
 |\sum_{i=1}^r |k_i|^2 + \frac {\ep_1 \sigma_{k_i}}{ \wki^{m}} - 
\sum_{i=1}^r |p_i|^2 + \frac {\ep_1 \sigma_{p_i}}{ \wpi^{m}} - b | \le \eta \}]\\
\le 2\eta \ep_1^{-1} \langle |n_-| \rangle ^{m}.
\end{align*}
Since $\mathbb P_{\bn, b}[ \dots ] = 0$ for  
$ |b| \le \frac 12 \left(\sum_{i=1}^{2r} |n_i|^2 + \frac { \ep_1}{ \wni^{m}}\right) + \eta $,
summing over $ |b| \le 2r ( |n_+|^2 + \ep_1 \langle n_-\rangle^{-m})$, 
we have 
\[ \mathbb P_{\bn} [\dots ] \le 4\eta(\bn) r(2\ep_1^{-1} |n_+|^2\langle n_-\rangle^{m} +1).\]
Therfore  for any non resonant $\bn = (n_1, \cdots, n_{2r})$ and for any $b\in \mathbb Z$, 
we have
\[ \mathbb P [ V = \{ w \in \mathcal W | |\Om(\bn) - b | \le \eta(\bn) \}] \le 
10  \ep_1^{-1}\sum_{ \bn = (n_1, \cdots, n_{2r}) } \eta(\bn)r |\bn_+|^2\langle n_-\rangle^{m}.\] 
The choice of $ \eta(\bn)  \le \ga \frac{\ep_1}{10 r}
 \langle \bn_+ \rangle^{-2(r+1)} \langle \bn_-\rangle^{-m)}$
gives 
\[ \mathbb P [V] \le \ga \sum_{ \bn = (n_1, \cdots, n_{2r}) } \langle \bn_+ \rangle ^{-2r} 
\le \sum_{ \bn = (n_1, \cdots, n_{2r}) } \Pi_{i=1}^{2r} \langle n_i\rangle^{-2} \le \ga.\]
Using $ \langle \bn_+ \rangle ^r \le r $, we set $\eta = \ga \frac{\ep_1}{10}
 \langle \bn_+ \rangle^{-3(r+1)} \langle \bn_-\rangle^{-m}$ and  conclude the lemma.
 \end{proof}
 \begin{proposition}
 Fix $\ga$ and $m > d/2$. There exists a set $F_{\ga}' \subset \mathcal W$
whose mearuse is larger than $1-\ga$ such that if 
$V \in F_{\ga}'$ then for any r 
\[ |\Om(\bn) + \la_1 w_{l_1} + \la_2 w_{l_2} | \ge
 \frac{\ep_1} {10} \ga  ( \ga/40)^{\frac{4r}{m}}
 \langle \bn_+ \rangle^ {- {4r^2}} \langle \bn_- \rangle^ { - m}\]
for any  $\bn = (n_1, \cdots, n_{2r})$, for any $l_1, l_2 \in \mathbb Z$, 
and for any $\la_1, \la_2 \in \{0, 1, -1\}$ 
such that $(\bn, l_1, l_2)$  is non resonant. 
 \end{proposition}
 \begin{proof}
 First if $\la_1 = \la_2 = 0,$ 
it holds trivially due to Lemma \ref{Lem}.\\
Secondly  $ \la_2 = 0, \la_1 = \pm 1:$ we note that 
\[  |\Om(\bn)| \le 3r|n_+|^2: = ( * ) .\]
If $ |l_1|^2 \ge 2 (*)$, we have 
\[ | \Om(\bn) + \la_1 w_{l_1}| > (*).\]
If $ |l_1|^2 \le  2 (*)$, we apply Lemma \ref{Lem} to $\bnp = (\bn, l_1)$ to have
\begin{align}\label{ii}
| \Om(\bn) + \la_1 w_{l_1}| \ge \ga \frac{\ep_1}{10} \langle \sqrt{6r} \bn_+ \rangle ^{-3(r+2)}
\langle \bn_- \rangle ^{-m} \ge \ga \frac{\ep_1}{10} 
\langle \bn_+ \rangle ^{-2r^2}\langle \bn_- \rangle ^{-m},
\end{align}
by using $r \le \langle \bn_+ \rangle^r$.\\
The case  $\la_1, \la_2$ has the same sign can be treated similarly.\\
 It remains to consider the form $ | \Om(\bn) + \la_1 w_{l_1}- w_{l_2}|$. 
We assume $|l_1| \le |l_2|$ wlog and further 
\begin{equation}\label{diff}
|l_2|^2 - |l_1|^2  \le 3(*)
\end{equation}
because $|\Om(\bn)| < (*)$ combining $|l_2|^2 - |l_1|^2  > 3(*)$ 
leads to $ | \Om(\bn) +  w_{l_1}- w_{l_2}| < 2(*)$.\\
By the tiangle inequality it holds that
\begin{align*}
| \Om(\bn) +  w_{l_1}- w_{l_2}| 
 \ge \left| | \Om(\bn) + |l_1|^2-|l_2|^2| - | w_{l_1}- w_{l_2} - (|l_1|^2-|l_2|^2)| \right|.
 \end{align*}
Note that 
\[ |w_{l_1} - w_{l_2} - l_1^2 + l_2^2| \le \left|  \frac{\ep_1 \sigma_{l_1}}{ \langle l_1 \rangle^m} -
\frac{\ep_1 \sigma_{l_2}}{ \langle l_2 \rangle^m}\right| \le \frac{2\ep_1}{\langle l_1 \rangle^m}.\]
Since $|l_1|^2-|l_2|^2 \in \mathbb Z$, by Lemma \ref{Lem} we have 
\[ | \Om(\bn) + |l_1|^2-|l_2|^2| \ge \ga \frac{\ep_1}{10}
 \langle \bn_+ \rangle^{-3(r+1)} \langle n_-\rangle^{-(1+m)}: = (**).\] 
So if $  \frac{2\ep_1}{\langle l_1 \rangle^m} \le \frac{(**)}{2}$, it leads to 
$ | \Om(\bn) +  w_{l_1}- w_{l_2}| \ge \frac{(**)}{2}$. 
Let us  consider the last case 
 $  \frac{2\ep_1}{\langle l_1 \rangle^m} \ge \frac{(**)}{2}$ under \eqref{diff}, 
 that is 
 \[ \langle l_1 \rangle < (40/\ga)^{\frac 1m} \langle \bn_+ \rangle^ {3(r+1)}{m} \langle n_-\rangle, 
 \quad 
 |l_2| < |l_1| + 2\sqrt{(*)}.\]
 Applying Lemma \ref{Lem} to $\max\{ |n_+|, |l_2|\} \le
 \{| n_+|, (40/\ga)^{\frac 1m} \langle \bn_+ \rangle^ {\frac {3(r+1)}{m}} \langle \bn_- \rangle + 2r |n_+|\}$ we conclude
 \begin{equation}\label{fin}
 | \Om(\bn) +  w_{l_1}- w_{l_2}| \ge  \frac{\ep_1} {10} \ga  ( \ga/40)^{\frac{3(r+1)}{m}}
 \langle \bn_+ \rangle^ {- \frac{3(r+1)^2}{m} } \langle \bn_- \rangle^ { - (r +m+1)}.
 \end{equation}
 Comparing the lower bounds \eqref{ii} and \eqref{fin} we have the proposition.
 \end{proof}

\end{document}